\documentclass[10pt,a4paper,DIV=classic]{myart}
\KOMAoptions{DIV=last}
\usepackage{todonotes}

\newcommand{\norm}[1]{\left\Vert#1\right\Vert}
\newcommand{\abs}[1]{\left\vert#1\right\vert}
\newcommand{\Set}[1]{\ensuremath{ \left\{ #1 \right\} }}
\newcommand{\set}[1]{\ensuremath{ \{ #1 \} }}

\newcommand*{\Cl}{\mathbf{cl}}
\newcommand*{\Int}{\mathbf{int}}
\newcommand*{\co}{{\bf conv}}
\newcommand*{\con}{{\rm s}}
\newcommand*{\cond}{{\bf s}}

\DeclareMathOperator*{\Lim}{\mathbf{Lim}}
\DeclareMathOperator*{\Limcstd}{\mathrm Lim}

\begin{document}

\title{The algebra of conditional sets and the concepts of conditional topology and compactness}
\ArXiV{}
\thanksColleagues{We thank Fares Maalouf, Stephan M\"uller and Martin Streckfu\ss\, for fruitful discussions. We wish to express our gratitude to an
anonymous referee for a careful reading of the manuscript and comments which essentially improved the presentation.}

\author[a,1,s]{Samuel Drapeau}
\author[b,2,t,v]{Asgar Jamneshan}
\author[c,3,u]{Martin Karliczek}
\author[b,4,v]{Michael Kupper}
\address[a]{CAFR and Department of Mathematics, Shanghai Jiao Tong University}
\address[b]{Department of Mathematics and Statistics, University of Konstanz}
\address[c]{Department of Mathematics, Humboldt University of Berlin}
\eMail[1]{drapeau@sjtu.edu.cn}
\eMail[2]{asgar.jamneshan@uni-konstanz.de}
\eMail[3]{karliczm@math.hu-berlin.de}
\eMail[4]{kupper@uni-konstanz.de}


\myThanks[s]{Funding: MATHEON project E.11}
\myThanks[t]{Funding: Berlin Mathematical School}
\myThanks[u]{Funding: Konsul Karl und Dr.~Gabriele Sandmann Stiftung}
\myThanks[v]{DFG Project KU 2740/2-1}
\date{\today}

\date{\today}

\abstract{
   The concepts of a conditional set, a conditional inclusion relation and a conditional
Cartesian product are introduced. The resulting conditional set theory is sufficiently
rich in order to construct a conditional topology, a conditional real and functional analysis indicating the possibility of a mathematical discourse based on conditional
sets. It is proved that the conditional power set is a complete Boolean algebra, 
and a conditional version of the axiom of choice, the ultrafilter lemma, Tychonoff's
theorem, the Borel-Lebesgue theorem, the Hahn-Banach theorem, the Banach-Alaoglu theorem and the Krein-\v{S}mulian theorem are shown. 
}

\keyWords{Conditional sets, conditional topology, conditional compactness, conditional functional analysis.}
\keyAMSClassification{03E70}

\maketitle

\section{Introduction}
Conditional set theory is an approach to study the local or dynamic behavior of structures whose local
or dynamic behavior is determined by the information encoded in a measure space, or more generally, in
a complete Boolean algebra. By constructing an analysis conditioned on a complete Boolean algebra, 
conditional set theory makes available analytic tools for this purpose. In the case of the associated measure 
algebra, it provides an alternative to measurable selection techniques, and extends the results in topological
$L^0$-modules obtained in Filipovi\'c et al.~\cite{kupper03} and Cheridito et al.~\cite{Cheridito2012}, 
initially motivated by dual representations of conditional risk measures \cite{kupper11}.

In the following, we briefly introduce conditional set theory. A \emph{conditional set} $\mathbf{X}$ is a collection of objects
$x|a$ for $x$ in a non-empty set $X$ and $a$ in a complete Boolean algebra $\mathcal{A}$ such that 
\begin{itemize}
        \item $a=b$ whenever $x|a=y|b$, 
        \item  $x|b=y|b$ implies $x|a=y|a$ for all $a,b \in\mathcal{A}$ with $a \leq b$, and   
        \item for any partition of unity $(a_i)$ in $\mathcal{A}$ and a family $(x_i)$ of elements in $X$ there exists exactly one element $x$ in $X$ such that $x|a_i=x_i|a_i$ for all $i$.  
\end{itemize} 
In order to introduce a conditional inclusion relation, it is necessary to specify conditional subsets of a
conditional set $\mathbf{X}$. A conditional subset of $\mathbf{X}$ is the collection of objects $\mathbf{Y}|b:= \{y|a\colon y \in Y,\, a \leq b\}$ for
some $b \in \mathcal{A}$ and some non-empty subset $Y$ of $X$ that is stable under pasting of its elements along partitions
of unity in $\mathcal{A}$. A conditional subset $\mathbf{Y}|b$ is a conditional set on the relative algebra $\mathcal{A}_b$. 
A conditional subset $\mathbf{Y}|b$ of $\mathbf{X}$ is conditionally included in another conditional subset $\mathbf{Z}|c$ if $\mathbf{Y}|b \subseteq \mathbf{Z}|c$. 
It can be shown that the collection of all conditional subsets of $\mathbf{X}$ together with the conditional inclusion relation forms a complete
Boolean algebra. The induced operations of conditional intersection, conditional union and conditional
complement conserve the structure of a conditional set, and satisfy the Boolean laws known from na\"ive set
theory. By giving meaning to a conditional Cartesian product, conditional relations and functions can be
defined as conditional subsets of the conditional product of two conditional sets, respectively.

A proof of a conditional version of a classical result is an adaptation of an existing classical proof. In
this adaptation process, it is helpful to recognize the following principles. The first principle is exhaustion
which establishes the largest condition $a$ for which a conditional property is satisfied. The second principle
is conditional negation which is stronger than classical negation. Conditional negation negates locally a
conditional property on all conditions $a > 0$. The third principle is localization. A conditional structure
on or a statement about a conditional set $\mathbf{X}$ can equivalently be stated on the conditional set $\mathbf{X}|b$ for any
$b < 1$ by passing from the complete Boolean algebra $\mathcal{A}$ to its complete relative algebra $\mathcal{A}_b$. In particular,
the restriction of a true statement about $\mathbf{X}$ to $\mathbf{X}|b$ remains true. The forth principle is bottom-up. It makes
a relation between a conditional concept on a conditional set $\mathbf{X}$ and its classical counterpart on the
underlying set $X$. For instance, we analyze this relation for the concepts of continuity and convergence in Section \ref{ch:topology}. 

Conditional set theory is closely related to the topos of sheaves over a complete Boolean algebra or 
Boolean-valued models of ZFC, respectively, see Mac Lane and Moerdijk \cite{lane1992sheaves} for an introduction to sheaves
in logic, see Bell \cite{bell2005set} and Kusraev and Kutateladze \cite{kusraev2012boolean} for an introduction to Boolean-valued models and
to Boolean-valued analysis, respectively, and see Jamneshan \cite{asgar13} for the connection of conditional sets to
sheaves and to Boolean-valued sets. Conditional set theory is an extension of the conditional analysis' results
for topological $L^0$-modules in \cite{Cheridito2012,DKKM13,kupper03,zapata14}. 
Conditional set operations on $(L^0)^d$ are introduced in Streckfu\ss \, \cite{martin11}. 
In this article, it is shown that $L^0$ is isomorphic to the conditional real numbers when the complete
Boolean algebra is the measure algebra associated to a $\sigma$-finite measure space. Hence the conditional analysis'
results obtained in $L^0$-theory are recovered in conditional set theory, and conditional set theory provides a
framework for further studies of stable $L^0$-modules. In Guo \cite{gou10}, conditional separation and duality results
for topological $L^0$-modules in \cite{kupper03} are related to the respective results in randomly normed modules, see
Haydon et al.~\cite{haydon1991randomly} and its references for an introduction to randomly normed spaces. In \cite{eisele2012,eisele13}, Eisele and
Taieb prove a conditional version of some classical theorems from functional analysis for modules over $L^\infty$.
Recently, a Hahn-Banach theorem for modules over Stonean algebras has been proved in Cerreia-Vioglio
et al.~\cite{kuppermaccheroni2014}. A conditional version of Mazur's lemma for $L^0$-modules is shown in Zapata-Garc\'ia \cite{zapata14}. Conditional
analysis is successfully applied to dynamic and conditional risk measures and decision theory in Filipovi\'c et
al.~\cite{kupper11}, Bielecki et al.~\cite{martinsamuel13}, Frittelli and Maggis \cite{frittelli14} and Jamneshan and Drapeau \cite{DJ14}, to backward stochastic
differential equations in Cheridito and Hu \cite{ch11} and Cheridito and Stadje \cite{cs12}, and to optimization problems
in equilibrium and principal-agent models in Horst et al.~\cite{kupper08} and Horst and Backhoff \cite{BH14}. 

This paper is organized as follows. Conditional sets, conditional set operations, conditional relations,
conditional families, conditional countability and a conditional axiom of choice are introduced in Section \ref{ch:set}.
In Section \ref{ch:topology}, conditional topological spaces and the concepts of conditional continuity, conditional convergence and conditional compactness are defined with the aim to prove a conditional version of Tychonoff's
theorem. In Section \ref{ch:real}, the conditional real line is constructed, and conditional metric spaces are defined
the conditional compactness of which is characterized by a conditional Borel-Lebesgue theorem. In Section
\ref{ch:vector}, conditional topological vector spaces are introduced and a conditional version of the Hahn-Banach theorem, the Banach-Alaoglu theorem and the Krein-\v{S}mulian theorem are shown. In the last two sections,
the main theorems are proved. A complete account can be found in Jamneshan \cite{jamneshan13} and Karliczek \cite{martindiss},
respectively.

\section{Conditional set theory}\label{ch:set}

Let $\mathcal{A}=(\mathcal{A},\wedge,\vee,{}^c,0,1)$ be a complete Boolean algebra. 
Examples are the power set algebra of some set, the measure algebra associated to a $\sigma$-finite measure space $(\Omega,\mathcal{F},\mu)$,\footnote{The associated measure algebra is the quotient Boolean algebra of $\mathcal{F}$ by the $\sigma$-ideal of $\mu$-null sets, see \cite[p.~233, Example 14.27]{monk1989handbook}} the quotient algebra $\mathcal{B}/\mathcal{I}$ where $\mathcal{B}$ is the Borel $\sigma$-algebra of a separable metric space and $\mathcal{I}$ the $\sigma$-ideal of meager sets, see \cite[p.~182, Proposition 12.9]{monk1989handbook}, and the Boolean algebra of all projective bands of a Riesz space, see \cite[Volume 3, p.~232, Theorem 352Q]{fremlin2000measure}.  
Recall that $\mathcal{A}$ together with the relation $a \leq b$ whenever $a \wedge b = a$ is a complete complemented distributive lattice. 
The relative algebra of $\mathcal{A}$ with respect to some $a\in\mathcal{A}$ is denoted by $\mathcal{A}_a:=\set{b\in\mathcal{A}: b\leq a}$. 
For a family $(a_i)=(a_i)_{i\in I}$ of elements in $\mathcal{A}$, its supremum is denoted by $\vee a_i=\vee_{i\in I} a_i$ and its infimum by $\wedge a_i=\wedge_{i\in I} a_i$.\footnote{As usual, we apply the conventions $\vee_\emptyset =0$ and $\wedge_\emptyset=1$.}
A partition of $a\in\mathcal{A}$ is a family $(a_i)$ of elements of $\mathcal{A}$ such that $a_i\wedge a_j=0$ whenever $i\neq j$ and $\vee a_i =a$.
Denote by $p(a)$ the set of all partitions of $a$.  
For every family $(a_i)_{i\in I}$ of elements in $\mathcal{A}$ there exists $(b_i)_{i\in I}\in p(\vee a_i)$ such that $b_i\leq a_i$ for all $i\in I$. Indeed, by the well-ordering theorem there exists a well-ordering on the index $I$, and define $b_i:=a_i\wedge (\vee_{j<i} b_j)^c$ for each $i\in I$. 
Equalities and inequalities between measurable functions are always understood in the almost sure sense whenever a measure is fixed. 
\begin{definition}\label{def:condset}
    A \emph{conditional set} $\mathbf{X}$ of a non-empty set $X$ and a complete Boolean algebra $\mathcal{A}$  is a collection of objects $x|a$ for $x\in X$ and $a\in\mathcal{A}$ such that 
    \begin{enumerate}[label=(C\arabic*)]
        \item\label{cond1} if $x|a=y|b$, then $a=b$;\footnote{In the first version of this paper, the assumption "identity" was required.
                This condition was dropped during the revision process, and we are grateful to Jos\'e Miguel Zapata Garc\'ia for pointing out to us the lack of this condition in the current version and suggesting the appropriate replacement.}
        \item\label{cond2} if $x,y\in X$ and $a,b \in\mathcal{A}$ with $a \leq b$, then $x|b=y|b$ implies $x|a=y|a$; 
        \item\label{cond3} if $(a_i)\in p(1)$ and $(x_i)$ is a family of elements in $X$, then there exists exactly one element $x\in X$ such that $x|a_i=x_i|a_i$ for all $i$.  
    \end{enumerate} 
    Condition \ref{cond2} is called \emph{consistency} and \ref{cond3} is named \emph{stability}. 
    For $(a_i)\in p(1)$ and a family $(x_i)$ of elements in $X$, the unique element $x\in X$ such that $x|a_i=x_i|a_i$ for all $i$ is the \emph{concatenation} of the family $(x_i)$ along the partition $(a_i)$, and denoted by $\sum x_i|a_i$. 
    For finite partitions, the concatenation is denoted by $x_{1}|a_1 + \ldots + x_{n}|a_n$.
\end{definition}
\begin{remark}
    Let $x,y \in X$ and $a \in \mathcal{A}$ such that $x|1=y|a$.
    Then it follows from \ref{cond1} that $a=1$ and from \ref{cond3} that $x=y$.
    In particular, $X$ is in bijection with $\set{x|1: x \in X}$.
    Furthermore, it follows from \ref{cond3} that $x|0=y|0$ for every $x,y \in X$.
    In particular, $\set{x|0: x \in X}$ consists of one element.
\end{remark}
\begin{examples}\label{ex:L0}
    \begin{enumerate}[fullwidth, label=\arabic*)]  
        \item Every conditional set can be identified with the collection of equivalence classes on the product $X\times \mathcal{A}$ for the equivalence relation $(x,a)\sim(y,b)$ whenever $x|a=y|b$.
        \item Let $\mathcal{A}=\set{0,1}$ be the trivial algebra and $X$ a non-empty set.
            The collection $\mathbf{X}$ of objects $x|1=x$ and $x|0=X\times \set{0}$ for all $x\in X$, is a conditional set.  
        \item Let $\mathcal{A}$ be a complete Boolean algebra.
            Then $\mathbf{X}=\mathcal{A}\times \mathcal{A}/_{\mathop \sim}$ where $(a,b)\sim (c,d)$ whenever $a\wedge b=c\wedge d$ and $b=d$, is a conditional set with equivalence classes $a|b$.
        \item\label{ex:l0} 
            Let $(\Omega,\mathcal{F},\mu)$ be a $\sigma$-finite measure space, $\mathcal{A}$ the associated measure algebra and $L^0$ the set of all equivalence classes of measurable functions $X:\Omega \to \mathbb{R}$ which coincide $\mu$-almost everywhere.
            Denote the equivalence classes in $\mathcal{F}$ by $a=[A]$ and the equivalence classes in $L^0$ by $x=[X]$.
            The collection $\mathbf{L}^0$ of objects $x|a=\set{y \in L^0: Y1_A=X1_A}$ is a conditional set.
        \item Conditional set of step functions: Let $\mathcal{A}$ be a complete Boolean algebra and $E$ a non-empty set.
            We consider the collection of all families $(x_i,a_i)$ of elements in $E\times \mathcal{A}$ where $(a_i)\in p(1)$.
            On this collection we define the equivalence relation $(x_i,a_i)\sim (y_j,b_j)$ if $\vee \set{a_i:x_i=z}=\vee \set{b_j:y_j=z}$ for all  $z \in E$, and we denote by $X$ the respective set of equivalence classes $[x_i,a_i]$.
            Inspection shows that we can make $X$ into a conditional set $\mathbf{X}$ by considering the collection of objects
            \begin{equation*}
                [x_i,a_i]|a:=\left\{ [y_j,b_j] \in X\colon \vee \set{a_i\colon x_i=z}\wedge a=\vee \set{b_j\colon y_j=z}\wedge a \text{ for all }z\in E\right\}.
            \end{equation*}
            This construction can be seen as the conditional set of step functions on $\mathcal{A}$ with values in $E$.
            Indeed, each $[x,1]$ can be uniquely identified with $x \in E$.
            Since by stability, elements in $X$ can be written as $\sum [x_i,1]|a_i$ it justifies the notation $\sum x_i|a_i$ for the elements of $X$ that can be interpreted as the step function taking the value $x_i\in E$ on $a_i$.
            
            In case that $E$ is either $\mathbb{N}$ or $\mathbb{Q}$ we denote the respective conditional set of step functions by $\mathbf{N}$ or $\mathbf{Q}$, and call them the \emph{conditional natural numbers} and \emph{conditional rational numbers}, respectively.
            The corresponding generating sets are denoted by $N=\{\sum n_i|a_i\colon (a_i) \in p(1), (n_i)\text{ is a family in }\mathbb{N}\}$ and $Q=\{\sum q_i|a_i\colon (a_i) \in p(1), (q_i)\text{ is a family in }\mathbb{Q}\}$.
    \end{enumerate}
\end{examples}
\begin{proposition}\label{rule02}
    Let $\mathbf{X}$ be a conditional set. 
    \begin{enumerate}[label=\textit{(\roman*)}]
        \item For all $(a_i),(b_j)\in p(1)$ and families $(x_{ij})$ of elements in $X$, it holds 
            \begin{equation*}
                \sum_j \Big(\sum_i x_{ij}|a_i\Big)|b_j = \sum_{i,j} x_{ij}|a_i\wedge b_j. 
            \end{equation*}
        \item For all $b\in\mathcal{A}$, $(a_i)\in p(b)$ and families $(x_i)$ of elements in $X$, there exists $x\in X$ such that $x|a_i = x_i|a_i$ for all $i$, and if $y\in X$ is such that $y|a_i = x_i|a_i$ for all $i$, then $x|b = y|b$.\footnote{Whenever there is no risk of confusion, we denote by $x|b=\sum x_i|a_i$.}
    \end{enumerate}
\end{proposition}  
\begin{proof}
    \begin{enumerate}[label=\textit{(\roman*)},fullwidth]
        \item Denote by $y_j:=\sum x_{ij}|a_i$ for each $j$ and $y:=\sum y_j|b_j$. 
            Since $y|b_j=y_j|b_j$, one has $y|a_i\wedge b_j=y_j|a_i\wedge b_j$ by consistency. 
            Similarly, it follows from $y_j|a_i=x_{ij}|a_i$ that $y_j|a_i\wedge b_j=x_{ij}|a_i\wedge b_j$. 
            Hence, $y|a_i\wedge b_j=x_{ij}|a_i\wedge b_j$ for all $i,j$, and thus $y=\sum x_{ij}|a_i\wedge b_j$. 
        \item Let $x,y\in X$ be such that $x|a_i=x_i|a_i$ for all $i$ and $x|b^c=w|b^c$, and $y|a_i=x_i|a_i$ for all $i$ and $y|b^c=z|b^c$ for some $w,z\in X$. 
            Let $v:=y|b + w|b^c$. 
            Since $v|b=y|b$, it holds $v|a_i=y|a_i$ for all $i$ due to consistency.  
             Since $y|a_i=x_i|a_i=x|a_i$ for all $i$, one has $v=x$ by stability. 
            By consistency, $x|b=v|b=y|b$. 
    \end{enumerate}
\end{proof}
\begin{definition} 
    Let $\mathbf{X}$ be a conditional set.
    A subset $Y$ of $X$ is called \emph{stable} if it is non-empty and 
    \begin{equation*}
        Y=\Set{\sum y_i|a_i: (a_i)\in p(1), (y_i)\text{ is a family of elements in } Y}.
    \end{equation*}
    Denote by $S(\mathbf{X})$ the set of all $Y\subseteq X$ which are stable.
\end{definition}
For every non-empty $Y\subseteq X$, the collection of those $\sum y_i|a_i$, where $(a_i)\in p(1)$ and $(y_i)$ is a family  of elements in $Y$, is the smallest stable set containing $Y$ due to Proposition \ref{rule02},  and is referred to as the \emph{stable hull} 
$\con(Y)$ of $Y$.
\begin{examples}\label{ex:sigmastable} 
    \begin{enumerate}[fullwidth, label=\arabic*)]
        \item For a conditional set $\mathbf{X}$ of a non-empty set $X$ and the trivial algebra $\mathcal{A}=\set{0,1}$, any non-empty subset of $X$ is stable.
        \item For any conditional set $\mathbf{X}$, every singleton $\set{x}$, $x \in X$, is stable. 
        \item For the conditional set $\mathbf{L^0}$, the set $[x,y]:=\set{z\in L^0: X\leq Z\leq Y}$ is stable for any $x,y\in L^0$ with $X\leq Y$. Further examples of stable subsets of $L^0$ can be found in \cite{kuppermaccheroni2014,Cheridito2012,DKKM13,kupper03}.  
    \end{enumerate}
\end{examples}  
By stability, every $Y \in S(\mathbf{X})$ generates a conditional set $\mathbf{Y}:=\set{y|a:y\in Y,a\in\mathcal{A}}$.
\begin{definition}
    Let $\mathbf{X}$ be a conditional set.
    Let $P(\mathbf{X})$ be the collection of all conditional sets $\mathbf{Y}$ generated by $Y \in S(\mathbf{X})$ and define
    \begin{equation*}
        \mathbf{P}(\mathbf{X}):=\left\{\mathbf{Y}|a=\left\{ y|b\colon y \in Y, b\leq a \right\}\colon \mathbf{Y} \in P(\mathbf{X}),a \in \mathcal{A}\right\},
    \end{equation*}
    which we call the \emph{conditional power set} of $\mathbf{X}$.
\end{definition}
Inspection shows that $\mathbf{P}(\mathbf{X})$ is a conditional set of $P(\mathbf{X})$.
Note that every element $\mathbf{Y}|a$ in the conditional power set of $\mathbf{X}$ is itself a conditional set of $Y|a:=\{y|a\colon y \in Y\}$ and the relative algebra $\mathcal{A}_a$, for the conditioning $(y|a)|b:=y|b$, $b\leq a$.
\begin{definition}
    Let $\mathbf{X}$ and $\mathbf{Y}$ be two conditional sets of $X, \mathcal{A}$ and $Y,\mathcal{B}$ respectively.
    We say that $\mathbf{Y}$ is \emph{conditionally included} in $\mathbf{X}$ and write $\mathbf{Y}\sqsubseteq \mathbf{X}$ if $\mathcal{B}=\mathcal{A}_a$ for some $a \in \mathcal{A}$ and $\mathbf{Y}=\mathbf{Z}|a$ for some $\mathbf{Z} \in P(\mathbf{X})$.
    We say that $\mathbf{Y}$ is a \emph{conditional subset} of $\mathbf{X}$ \emph{on} $a$.
\end{definition}
Depending on the context, we write $\mathbf{Y}$ or $\mathbf{Y}|a$ for a conditional subset of $\mathbf{X}$ on $a$.
Note that 
\begin{equation*}
  \mathbf{P}(\mathbf{X})=\left\{\mathbf{Y}\sqsubseteq \mathbf{X}\colon \mathbf{Y}\text{ conditional set}\right\}.
\end{equation*}
By inspection, $\sqsubseteq$ is a partial order on $\mathbf{P}(\mathbf{X})$ with greatest element $\mathbf{X}=\mathbf{X}|1$ and least element $\mathbf{X}|0$.
Every singleton $\{x\}$, $x \in X$, is stable and defines a conditional set $\mathbf{x}:=\{x|a\colon a \in \mathcal{A}\}$ called a \emph{conditional element of} $\mathbf{X}$.\footnote{Although the conditional elements $\mathbf{x}$ of $\mathbf{X}$ are formally sets, they are atoms among the conditional subsets of $\mathbf{X}$ on $1$, and are in one-to-one relation with $x \in X$. In the case that $\mathbf{Y}$ is a conditional set on $a$, a conditional element $\mathbf{y}$ of $\mathbf{Y}$ is on $a$.}
Denote by $\cond(Y)\in P(\mathbf{X})$ the conditional set generated by the stable hull of some non-empty subset $Y\subseteq X$.
\begin{theorem}\label{thm:powerset}
    Let $\mathbf{X}$ be a conditional set. 
    Then $(\mathbf{P}(\mathbf{X}),\sqsubseteq)$ is a complete complemented distributive lattice. 
\end{theorem}
\begin{proof}  
    First we prove completeness, second complementation and third distributivity.   
    \begin{enumerate}[label=\textit{Step \arabic* :},fullwidth, ref=\textit{Step \arabic*}]
        \item \label{step1} Let $(\mathbf{Y}_i|a_i)$ be a non-empty family of conditional subsets of $\mathbf{X}$. 
            First we construct the supremum of $(\mathbf{Y}_i|a_i)$ and second its infimum. 
            Fix $z\in X$, and let 
            \begin{equation*}
                Y=\Set{\sum y_i|b_i + z|\wedge a_i^c: (b_i)\in p(\vee a_i)\text{ with } b_i \leq a_i, y_i \in Y_i \text{ for each }i}   
            \end{equation*} 
            where we used the well-ordering theorem to find a partition $(b_i)\in p(\vee a_i)$ with $b_i\leq a_i$ for all $i$, and Proposition \ref{rule02} to construct each of the concatenations $\sum y_i|b_i + z|\wedge a_i^c$. 
            We want to show that $Y$ is stable.  
            To this end, let $(c_j) \in p(1)$ and $(x_j)$ be a family of elements in $Y$ where $x_j=\sum y_{ij}|b_{ij}+z|\wedge a_i^c$ for each $j$.  
            For each $i$, define $b_i=\vee b_{ij}\wedge c_j$ and $y_i=\sum y_{ij}|c_j$. 
            Inspection shows that $(b_i)\in p(\vee a_i)$ with $b_i\leq a_i$.
            By stability, one has $y_i\in Y_i$ for each $i$. 
            From Proposition \ref{rule02} it follows that
            \begin{equation*}
                \sum x_j|c_j=\sum y_i|b_i + z|\wedge a_i^c\in Y.   
            \end{equation*}
            With $\mathbf{Y}$ being the conditional set generated by the stable set $Y$, we show that 
            \begin{equation}\label{eq:condunion}
                \sqcup \mathbf{Y}_i|a_i:=\mathbf{Y}|\vee a_i 
            \end{equation}
            is the supremum of $(\mathbf{Y}_i|a_i)$.  
            Indeed, for any $i$, $y_i\in Y_i$ and some arbitrary $y\in Y$, it holds $w:=y_i|a_i + y|a_i^c\in Y$ due to Proposition \ref{rule02}. 
            Since $w|a_i=y_i|a_i$, it follows that $\mathbf{Y}|\vee a_i$ is an upper bound due to consistency. 
            For any other upper bound $\mathbf{W}|c$, it must hold $a_i\leq c$ for each $i$, and therefore $\vee a_i\leq c$. 
            Moreover, for all $i$ and every $y_i\in Y_i$ there is $w\in W$ such that $y_i|a_i=w|a_i$.  
            By Proposition \ref{rule02}, for all $y\in Y$ there is $w\in W$ with $y|\vee a_i=w|\vee a_i$. 
            By consistency, $\mathbf{Y}|\vee a_i$ is the least upper bound. 

            We want to show that there exists an infimum of $(\mathbf{Y}_i|a_i)$. 
            Let 
            \begin{equation*}
                M=\set{a: a\leq \wedge a_i, \text{ there exists }x\in X \text{ such that for all $i$ there is }y_i\in Y_i\text{ with }x|a=y_i|a},
            \end{equation*}
            and $b=\vee M$.
            We show that $b$ is attained. 
            Let $(c_j)$ be a family of elements in $M$ with $\vee c_j=b$, that is for each $j$ there exists $x_j\in X$ and for all $i$ there exists $y_{ij}\in Y_i$ with $x_j|c_j=y_{ij}|c_j$.  
            By the well-ordering theorem, there is $(d_j)\in p(b)$ such that $d_j\leq c_j$ for all $j$. 
            By consistency, $d_j\in M$ for each $j$. 
            By Proposition \ref{rule02}, one has $\sum x_j|d_j=\sum y_{ij}|d_j$ for all $i$, and thus $b\in M$. 
            Next we show that 
            \begin{equation*}
                Y:=\set{x\in X: \text{ for all $i$ there exists } y_i\in Y_i \text{ with } x|b=y_i|b} 
            \end{equation*}
            is stable. 
            To this end, let $(c_j)\in p(1)$ and $(x_j)$ be a family of elements in $Y$. 
            For all $j$ and every $i$ there is $y_{ij}\in Y_i$ with $x_j|b=y_{ij}|b$. 
            Set $y_i=\sum y_{ij}|c_j$ for each $i$. 
            By Proposition \ref{rule02}, it holds $y_i|b=x_j|b$, and thus $(\sum x_j|c_j)|b=y_i|b$. 
            By construction, 
            \begin{equation}\label{eq:intersection}
                \sqcap \mathbf{Y}_i|a_i:=\mathbf{Y}|b
            \end{equation}
            is a lower bound of $(\mathbf{Y}_i|a_i)$. 
            For any other lower bound $\mathbf{W}|c$, it holds $c\leq a_i$ for all $i$. 
            Moreover, for all $w\in W$ and every $i$ there exists $y_i\in Y_i$ such that $w|c=y_i|c$.  
            Therefore $c\leq b$. 
            By consistency, it holds $\set{w|d:w\in W,d\leq c}\subseteq \set{y|d:y\in Y, d\leq c}$. 
            Thus $\sqcap \mathbf{Y}_i|a_i$ is the greatest lower bound of $(\mathbf{Y}_i|a_i)$. 
        \item We want to show that for all conditional subsets $\mathbf{Y}|a$ of $\mathbf{X}$ it holds 
            \begin{equation*}
                \mathbf{Y}|a\sqcap (\mathbf{Y}|a)^\sqsubset =\mathbf{X}|0 \quad \text{and} \quad \mathbf{Y}|a\sqcup (\mathbf{Y}|a)^\sqsubset =\mathbf{X}|1, 
            \end{equation*}
            where 
            \begin{equation}\label{eq:complement}
                (\mathbf{Y}|a)^\sqsubset:=\sqcup \Set{\mathbf{Z}|c\in\mathbf{P}(\mathbf{X}): \mathbf{Z}|c\sqcap \mathbf{Y}|a=\mathbf{X}|0} 
            \end{equation}
            is the complement. 
            By completeness which has been proved in \ref{step1}, it holds $(\mathbf{Y}|a)^\sqsubset \in \mathbf{P}(\mathbf{X})$.
            Suppose $(\mathbf{Y}|a)^\sqsubset$ is of the form $\mathbf{W}|b$ for some $\mathbf{W}\in P(\mathbf{X})$ and $b\in\mathcal{A}$. 

            As for the first statement, by way of contradiction, we may assume that $\mathbf{Y}|a\sqcap \mathbf{W}|b=\mathbf{Z}|c$ for some $c>0$.  
            Thus there exists $y\in Y$ and $w\in W$ such that $y|c=w|c$. 
            However, this implies $\mathbf{y}|c$ satisfies $\mathbf{y}|c\sqcap \mathbf{Y}|a\neq \mathbf{X}|0$ which is contradictory.   
            Hence $c=0$, and thus $\mathbf{Y}|a\sqcap \mathbf{W}|b=\mathbf{X}|0$.  

            As for the second statement, we may assume that $\mathbf{Y}|a\neq \mathbf{X}|1$, since otherwise $\mathbf{Y}|a\sqcup \mathbf{W}|b=\mathbf{X}|1\sqcup \mathbf{X}|0=\mathbf{X}|1$. 
            Suppose $\mathbf{Y}|a\sqcup \mathbf{W}|b=\mathbf{Z}|a\vee b$ for some $\mathbf{Z}\in P(\mathbf{X})$, and let $x\in X$. 
            Then $\mathbf{x}\sqcap \mathbf{Y}|a=\mathbf{y}|c$ for some $y\in Y$ and $c=\vee\set{a^\prime: a^\prime \leq a, \text{ there exists } y\in Y \text{ with } x|a^\prime=y|a^\prime}$ and $\mathbf{x}\sqcap \mathbf{W}|b=\mathbf{w}|d$ for some $w\in W$ and $d=\vee\set{b^\prime: b^\prime \leq b, \text{ there exists } w\in W \text{ with } x|b^\prime=w|b^\prime}$. 
            Since $\mathbf{Y}|a\sqcap \mathbf{W}|b=\mathbf{X}|0$, it must hold $c\wedge d=0$ and $c\vee d=1$.\footnote{By inspection, the negation of either of the two conditions leads immediately to a contradiction to the construction of the infimum $\sqcap$ and the complement ${}^\sqsubset$.}  
            It follows from the construction of the supremum $\sqcup$ that $a\vee b=1$ and $x=x|c + x|d\in Z$. 
            Thus $\mathbf{Z}|a\vee b=\mathbf{X}|1$. 
        \item Let $\mathbf{Y}_k|a_k\in \mathbf{P}(\mathbf{X})$ for $k=1,2,3$. 
            Since it has already been shown that $(\mathbf{P}(\mathbf{X}),\leq)$ is a lattice, both distributive laws are equivalent, see \cite[p.~15, Lemma 1.17]{monk1989handbook}. 
            It remains to prove that 
            $$(\mathbf{Y}_1|a_1\sqcap \mathbf{Y}_2|a_2)\sqcup (\mathbf{Y}_1|a_1\sqcap \mathbf{Y}_3|a_3) = \mathbf{Y}_1|a_1 \sqcap (\mathbf{Y}_2|a_2\sqcup \mathbf{Y}_3|a_3).$$
            On the one hand, in every lattice it holds $$(\mathbf{Y}_1|a_1\sqcap \mathbf{Y}_2|a_2)\sqcup (\mathbf{Y}_1|a_1\sqcap \mathbf{Y}_3|a_3) \leq \mathbf{Y}_1|a_1 \sqcap (\mathbf{Y}_2|a_2\sqcup \mathbf{Y}_3|a_3).$$

            On the other hand, suppose $(\mathbf{Y}_1|a_1\sqcap \mathbf{Y}_2|a_2)\sqcup (\mathbf{Y}_1|a_1\sqcap \mathbf{Y}_3|a_3)=\mathbf{V}|c$ for some $\mathbf{V}\in P(\mathbf{X})$ and $c\in\mathcal{A}$. 
            Without loss of generality, we may assume that $\mathbf{Y}_1|a_1 \sqcap (\mathbf{Y}_2|a_2\sqcup \mathbf{Y}_3|a_3)$ is of the form $\mathbf{W}|1$ for some $\mathbf{W}\in P(\mathbf{X})$.  
            By the construction of the infimum $\sqcap$, this immediately implies $a_1=1$ and $\mathbf{Y}_2|a_2\sqcup \mathbf{Y}_3|a_3=\mathbf{Z}|1$ for some $\mathbf{Z}\in P(\mathbf{X})$. 
            Moreover, for every $w\in W$ there are $y\in Y_1$ and $z\in Z$ such that $w|1=y|1=z|1$.    
            By the construction of the supremum $\sqcup$, there exists $(b,b^c)\in p(1)$ with $b\leq a_2$ and $b^c\leq a_3$, and $v\in Y_2$ and $u\in Y_3$ such that $z=v|b+ u|b^c$.  
            By consistency, $w|b=y|b=z|b=v|b$ and $w|b^c=y|b^c=z|b^c=u|b^c$. 
            Since $w=w|b+ w|b^c$, it follows $w\in V$ and $c=1$.   
            Thus $\mathbf{W}|1\sqsubseteq \mathbf{V}|1$ which finishes the proof. 
    \end{enumerate}
\end{proof}
The operations $\sqcup, \sqcap$ and ${}^\sqsubset$, given in \eqref{eq:condunion}, \eqref{eq:intersection} and \eqref{eq:complement}, are named \emph{conditional union, intersection} and \emph{complement}, respectively. 
As a consequence of standard results on complete complemented distributive lattices and Boolean algebras, see \cite[p.~14]{monk1989handbook}, we have: 
\begin{corollary}
    For every conditional set $\mathbf{X}$,
    \begin{equation*}
        \mathbf{P}(\mathbf{X})=(\mathbf{P}(\mathbf{X}),\sqcup,\sqcap,{}^\sqsubset, \mathbf{X}|0,\mathbf{X})   
    \end{equation*}
    is a complete Boolean algebra.
\end{corollary}
\begin{remark}
    The complete Boolean algebra $\mathbf{P}(\mathbf{X})$ is atomic if, and only if, $\mathcal{A}$ is atomic. 
    Indeed, let $A$ be the set of atoms of $\mathcal{A}$. 
    Then the set of atoms of $\mathbf{P}(\mathbf{X})$ is $\set{\mathbf{x}|a: a \in A, x\in X}$.
    Conversely, if $\mathcal{A}$ is atomless then for each $a >0 $ there exists $0<b<a$ such that $\mathbf{x}|b \sqsubseteq \mathbf{x}|a$ and $\mathbf{x}|b\neq \mathbf{x}|a$ for all $x\in X$.
    Analogously, one can verify that the distributivity law of $\mathbf{P}(\mathbf{X})$ coincides with the distributive law of $\mathcal{A}$. 
\end{remark} 
\begin{corollary}
    Let $\mathbf{X}$ be a conditional set.  
    \begin{enumerate}[label=(\roman*)]
        \item \emph{De Morgan's law}: For any non-empty family $(\mathbf{Y}_i)$ of conditional subsets of $\mathbf{X}$, it holds $$(\sqcup \mathbf{Y}_i)^\sqsubset=\sqcap (\mathbf{Y}_i)^\sqsubset.$$  
        \item \emph{Distributivity}: For any non-empty family $(\mathbf{Y}_{ij})_{i\in I, \, j\in J}$ of conditional subsets of $\mathbf{X}$ where $J$ is arbitrary and $I$ is finite, it holds $$\sqcap_{i\in I} \sqcup_{j\in J} \mathbf{Y}_{ij} = \sqcup\Set{\sqcap_{i\in I}\mathbf{Y}_{if(i)}: f\in J^I}.$$ 
        \item \emph{Associativity}: For any non-empty family $(\mathbf{Y}_{ij})_{i\in I, \, j\in J}$ of conditional subsets of $\mathbf{X}$ where $I,J$ are arbitrary, it holds $$\sqcup_{i\in I} (\sqcup_{j\in J} \mathbf{Y}_{ij})=\sqcup_{i\in I,\, j\in J} \mathbf{Y}_{ij}.$$  
    \end{enumerate} 
\end{corollary} 
\begin{proof}
    All three properties are satisfied in every complete Boolean algebra, see \cite[p.~22, Lemma 1.33]{monk1989handbook}. 
\end{proof}
\begin{lemma}\label{lem:1fcunion}
    Let $\mathbf{X}$ be a conditional set.
    \begin{enumerate}[label=(\roman*)]
        \item For $\mathbf{Y}_1,\mathbf{Y}_2 \sqsubseteq \mathbf{X}$ such that $\mathbf{Y}_1\sqcap \mathbf{Y}_2$ is on $1$, one has $Y_1\cap Y_2\neq \emptyset$. 
        \item For $\mathbf{Y}_1,\mathbf{Y}_2\sqsubseteq \mathbf{X}$ on $1$ such that $\mathbf{Y}_1\sqsubseteq \mathbf{Y}_2$, one has $\mathbf{Y}_1\subseteq \mathbf{Y}_2$ and $Y_1\subseteq Y_2$. 
        \item Let $(\mathbf{Y}_i)$ be a stable collection of conditional subsets of $\mathbf{X}$ each on $1$. Then one has $\sqcup \mathbf{Y}_i=\mathbf{W}$ where $W=\cup Y_i$. 
    \end{enumerate}
\end{lemma}
\begin{proof}
    The first two statements are immediate from the definitions. 
    As for the third one, let 
    \begin{equation*}
        W=\Set{\sum y_j|a_j: (a_j)\in p(1), y_j\in Y_{i_j} \text{ for some }i_j}.  
    \end{equation*}
    It follows that $\cup Y_i\subseteq W$.
    Conversely, let $y=\sum y_j|a_j \in W$, so that $y \in \sum Y_{i_j}|a_j$.
    By stability, $\sum Y_{i_j}|a_j=Y_{i_k}$ for some $i_k$ and therefore $y \in \cup Y_i$ showing that $W=\cup Y_i$.
    By definition of the conditional union, it holds $\sqcup \mathbf{Y}_i=\mathbf{W}$. 
\end{proof}

\subsection{Conditional relation and function}

\begin{definition}\label{d:product}
    Let $I$ be a non-empty index set and $\mathbf{X}_i$ be a conditional set for each $i\in I$. 
    The collection of objects $(x_i)_{i\in I}|a:=(x_i|a)_{i\in I}$ for $(x_i)_{i\in I}\in \prod_{i\in I} X_i$ and $a\in\mathcal{A}$ is called the \emph{conditional product} of the $\mathbf{X}_i$, and is denoted by $ \prod_{i\in I} \mathbf{X}_i$.
\end{definition}
Note that the conditional product is a conditional set.
For a conditional set $\mathbf{X}$ and a non-empty index set $I$, we write $\mathbf{X}^I:=\prod_{i\in I}\mathbf{X}$ and $\mathbf{X}^n:=\prod_{1\leq k\leq n} \mathbf{X}$ for any $n\in\mathbb{N}$. 
\begin{definition}
    Let $\mathbf{X}$ and $\mathbf{Y}$ be conditional sets. 
    A \emph{conditional binary relation} on $\mathbf{X}\times \mathbf{Y}$ is a conditional subset $\mathbf{T}$ of $\mathbf{X}\times \mathbf{Y}$ on $1$.\footnote{A conditional binary relation $\mathbf{T}$ is a classical relation $T\subseteq X\times Y$ such that $\sum (x_i,y_i)|a_i=(\sum x_i|a_i,\sum y_i|a_i)\in T$ for all $(a_i)\in p(1)$ and every family $(x_i,y_i)$ of elements in $T$.}
    We say that $\mathbf{x}|a$ and $\mathbf{y}|a$ for two conditional elements $\mathbf{x},\mathbf{y}$ of $\mathbf{X}\times \mathbf{Y}$ are in relation if $(x|a,y|a) \in \mathbf{T}$ and write $\mathbf{x}|a\,\mathbf{T}\,\mathbf{y}|a$.
    A conditional binary relation $\mathbf{T} $ on $\mathbf{X}\times\mathbf{X}$ is \emph{conditionally antisymmetric, reflexive, symmetric} or \emph{transitive} whenever $T$ is antisymmetric, reflexive, symmetric or transitive in the classical sense.   
    A \emph{conditional partial order} $\leqslant$ is a conditionally antisymmetric, reflexive and transitive binary relation. 
    Given a conditional partial order $\leqslant$, we define $\mathbf{x}|a < \mathbf{y}|a$ if $\mathbf{x}|a\leqslant \mathbf{y}|a$ and $\mathbf{x}|b\neq \mathbf{y}|b$ for every $0<b\leq a$.\footnote{By consistency, if $\mathbf{x}|a \leqslant \mathbf{y}|a$ it holds $\mathbf{x}|b\leqslant \mathbf{y}|b$ for every $b\leq a$.}
    A conditional partial order $\leqslant$ is \emph{conditionally total} if for every $x,y \in X$ there exists $(a,b,c)\in p(1)$ such that $\mathbf{x}|a< \mathbf{y}|a$, $\mathbf{y}|b< \mathbf{x}|b$ and $\mathbf{x}|c =\mathbf{y}|c$. 
    A \emph{conditional equivalence relation} is a conditionally symmetric, reflexive and transitive binary relation.
\end{definition}
Conditional extrema of a conditionally partially ordered set are defined as classical extrema.
For instance, a conditional direction is a conditionally reflexive and transitive binary relation with the property that every pair $x,y\in X$ has an upper bound.
\begin{examples}\label{ex:order}
    \begin{enumerate}[fullwidth,label=\arabic*)]
        \item On $Q$ define $\sum p_i|a_i\leq \sum q_j|b_j$ whenever $p_i\leq q_j$ for all $i,j$ with $a_i\wedge b_j>0$. 
            It is immediate from the definition, that the previous binary relation is a stable subset of $Q\times Q$. 
            The related conditional relation $\leqslant$ on the conditional rational numbers $\mathbf{Q}$ is a conditionally total partial order.  
        \item On $L^0$ define $x\leq y$ whenever $X\leq Y$. 
            By inspection, this binary relation is a stable subset of $L^0\times L^0$. 
            The related conditional relation $\leqslant$ on $\mathbf{L^0}$ is a conditionally total partial order.     
    \end{enumerate}
\end{examples} 
\begin{definition}\label{def:function}
    Let $\mathbf{X}$ and $\mathbf{Y}$ be conditional sets. 
    A function $f:X\to Y$ is \emph{stable} if 
    \begin{equation*}
        f\Big(\sum x_i|a_i\Big)=\sum f(x_i)|a_i, \quad \text{ for all }(a_i)\in p(1) \text{ and every family }(x_i)\text{ of elements in } X.
    \end{equation*}
    A conditional subset $\mathbf{G}_{\mathbf{f}}$ of $\mathbf{X}\times\mathbf{Y}$ on $1$ is the \emph{graph of a conditional function} $\mathbf{f}:\mathbf{X}\to\mathbf{Y}$ whenever $G_f$ is the graph of a stable function $f:X\to Y$. 
    
    For $\mathbf{U}|b\sqsubseteq \mathbf{X}$, the \emph{conditional image} $\mathbf{f}(\mathbf{U}|b)$ is defined as $\mathbf{Z}|b\sqsubseteq \mathbf{Y}$ where $Z=\set{f(x)\colon x\in U}$. For $\mathbf{V}|c\sqsubseteq \mathbf{Y}$, the \emph{conditional preimage} $\mathbf{f}^{-1}(\mathbf{V}|c)$ is defined as $\mathbf{W}|d\sqsubseteq \mathbf{X}$ where\footnote{The stability of $W$ follows from the stability of $f$. Note that $d=c$ whenever $\mathbf{f}(\mathbf{X}|c)=\mathbf{Y}|c$.}
    \begin{align*}
       d&=\vee\set{a\colon a\leq c, \text{ there are }x\in X, y\in V \text{ such that } f(x)|a=y|a}  \\
       W&=\set{x\in X\colon \text{there is } y\in V \text{ with } f(x)|d=y|d}.
    \end{align*}
    A conditional function $\mathbf{f}:\mathbf{X}\to\mathbf{Y}$ is \emph{conditionally injective} if $x,x^\prime\in X$ with $x|a\neq x^\prime|a$ for all $a\neq 0$ implies $f(x)|a\neq f(x^\prime)|a$ for all $a\neq 0$; it is \emph{conditionally surjective} if $f$ is surjective; and it is \emph{conditionally bijective} if it is conditionally injective and surjective. 
\end{definition}
\begin{examples}\label{exep:Qfunc}
    \begin{enumerate}[fullwidth, label=\arabic*)]
        \item Let $X$ and $Y$ be non-empty sets and $f:X\to Y$ an injective function. 
            Let $\mathbf{X}$, $\mathbf{Y}$ and $\mathbf{G_f}$ be the conditional sets generated by $X,Y$ and the graph $G_f$ of $f$.\footnote{In the sense of Example \ref{ex:L0} 5).}
            Then the conditional function $\mathbf{f}:\mathbf{X}\to\mathbf{Y}$ is conditionally injective. 
        \item Let $\prod \mathbf{X}_i$ be the conditional product of a family of conditional sets.  
            The $j$-th projection $pr_j:\prod X_i\to X_j$ is a stable function, and $\mathbf{pr}_j:\prod \mathbf{X}_i\to \mathbf{X}_j$ is called the \emph{$j$-th conditional projection}. 
        \item Let $\mathbf{X}$ be a conditional set and $\mathbf{Y}$ a conditional subset of $\mathbf{X}$ on $1$. The embedding $Y\hookrightarrow X$ is stable, and the conditional function $\mathbf{Y}\hookrightarrow \mathbf{X}$ is called a \emph{conditional embedding}. 
        \item We call the conditional function $\abs{\cdot}:\mathbf{Q}\to \mathbf{Q}_+=\{\mathbf{q}\colon \mathbf{0}\leqslant \mathbf{q}\}$ generated by $q=\sum q_i|a_i\mapsto \sum \abs{q_i}|a_i$ on $Q$, the \emph{conditional absolute value}.
    \end{enumerate}
\end{examples}
The assertions in the following proposition follow from Theorem \ref{thm:powerset}. 
\begin{proposition}\label{prop:condfunc}
    Let $\mathbf{X}$ and $\mathbf{Y}$ be two conditional sets and $\mathbf{f}:\mathbf{X}\to \mathbf{Y}$ a conditional function. 
    For a non-empty family $(\mathbf{U}_i)$ of conditional subsets of $\mathbf{X}$, a non-empty family $(\mathbf{V}_j)$ of conditional subsets of $\mathbf{Y}$, conditional subsets $\mathbf{Z}_1,\mathbf{Z}_2$ of $\mathbf{X}$ with $\mathbf{Z}_1\sqsubseteq \mathbf{Z}_2$ and conditional subsets $\mathbf{W}_1$, $\mathbf{W}_2$ of $\mathbf{Y}$ with $\mathbf{W}_1\sqsubseteq \mathbf{W}_2$, and $\mathbf{Z}\sqsubseteq \mathbf{X}$ and $\mathbf{W}\sqsubseteq \mathbf{Y}$, it holds 
    \begin{align}
        \mathbf{f}(\mathbf{Z}_1)&\sqsubseteq \mathbf{f}(\mathbf{Z}_2),&\mathbf{f}^{-1}(\mathbf{W}_1)&\sqsubseteq \mathbf{f}^{-1}(\mathbf{W}_2),\label{f:021}\\
        \mathbf{f}(\sqcup \mathbf{U}_i)&=\sqcup \mathbf{f}(\mathbf{U}_i),&\mathbf{f}^{-1}(\sqcup \mathbf{V}_i)&=\sqcup \mathbf{f}^{-1}(\mathbf{V}_i),\label{f:031}\\
        \mathbf{f}(\sqcap \mathbf{U}_i)&\sqsubseteq \sqcap \mathbf{f}(\mathbf{U}_i),&\mathbf{f}^{-1}(\sqcap \mathbf{V}_i) &= \sqcap \mathbf{f}^{-1}(\mathbf{V}_i),\label{f:041}\\
        \mathbf{f}(\mathbf{Z})^\sqsubset \sqcap \mathbf{f}(\mathbf{X}) &\sqsubseteq \mathbf{f}(\mathbf{Z}^{\sqsubset}),&\mathbf{f}^{-1}(\mathbf{W}^\sqsubset)&=\mathbf{f}^{-1}(\mathbf{W})^{\sqsubset},\label{f:051}\\
        \mathbf{Z}&\sqsubseteq \mathbf{f}^{-1}(\mathbf{f}(\mathbf{Z})),&\mathbf{f}(\mathbf{f}^{-1}(\mathbf{W}))&\sqsubseteq \mathbf{W}.\label{f:061}
    \end{align}
    It holds equality on the left-hand side of \eqref{f:061} if $\mathbf{f}$ is conditionally injective, and on its right-hand side if $\mathbf{W}\sqsubseteq \mathbf{f}(\mathbf{X})$.  
\end{proposition}
\subsection{Conditional family and axiom of choice}
\begin{definition}
    Let $\mathbf{X}$ and $\mathbf{I}$ be conditional sets.
    A \emph{stable family} $(x_i)$ of elements in $X$ is the graph $G_f=\{(f(i),i)\colon i\in I\}$ where $\mathbf{f}:\mathbf{I}\to \mathbf{X}$ is a conditional function.
    In particular, $\sum x_{i_j}|a_j=x_{\sum i_j|a_j}$ for every $(a_j)\in p(1)$ and every family $(i_j)$ of elements in $I$. 
    A \emph{conditional family} $(\mathbf{x}_{\mathbf{i}})$ of conditional elements of $\mathbf{X}$ is the conditional graph $\mathbf{G}_\mathbf{f}$.
\end{definition} 
A conditional family $(\mathbf{x_i})$ is a conditional element of $\prod_{i\in I}\mathbf{X}_i$ where $\mathbf{X}_i=\mathbf{X}$ for each $i\in I$. 
By Proposition \ref{rule02}, the collection of all stable families $(x_i)$ is a stable subset of $\prod_{i\in I} X_i$, and we denote by $\mathbf{X^I}$ the conditional set corresponding to it. 
For a conditional natural number $\mathbf{n}$, we write $\mathbf{X^n}$ for $\mathbf{X^{\{1\leqslant l\leqslant n\}}}$. 
Note that $$\Set{\mathbf{1}\leqslant \mathbf{l} \leqslant \mathbf{n}} = \cond\left(\Set{\sum l_i|a_i \colon 1\leq l_i\leq n_i, \text{ for each }i}\right)$$ where $n=\sum n_i|a_i\in N$.
In particular, stability implies $\mathbf{X^n}=\sum \mathbf{X}^{n_i}|a_i$. 
\begin{example}
    Let $\mathbf{X}$ be a conditional set and $(\mathbf{I},\leqslant)$ be a conditional direction.
    A conditional family $(\mathbf{x}_{\mathbf{i}})$ of conditional elements of $\mathbf{X}$ is called a \emph{conditional net}.  
    In case that $(\mathbf{I},\leqslant)$ equals to $(\mathbf{N},\leqslant)$, a conditional family $(\mathbf{x}_{\mathbf{n}})$ is called a \emph{conditional sequence} of conditional elements of $\mathbf{X}$.
\end{example}
\begin{lemma}\label{lem:condfinite}
    Let $\mathbf{X}$ be a conditional subset of $\mathbf{N}$ on $1$.
    Suppose that $\mathbf{x}\leqslant \mathbf{k}$ for every conditional element $\mathbf{x}$ of $\mathbf{X}$ and for some conditional element $\mathbf{k}$ of $\mathbf{N}$.
    Then there exists a conditional bijection $\mathbf{f}:\mathbf{X}\to \set{\mathbf{1}\leqslant \mathbf{l} \leqslant \mathbf{n}}$ for a unique conditional element $\mathbf{n}$ of $\mathbf{N}$.   
\end{lemma}
\begin{proof}
    Suppose that $k=\sum k_i|b_i$.
    For each $i$, let $J_i$ be the set of non-empty subsets of $\set{1,\ldots,k_i}$. 
    For each $M_i\in J_i$, define $a_{M_i}=\vee\{a\colon a\leq b_i\text{ and }\{n \in N\colon x|a=n|a\text{ for some }x \in X\}=M_i\}$.
    Then one has $(a_{M_i})\in p(b_i)$ for each $i$.
    Define $n=\sum_{i, M_i\in J_i} \text{card}(M_i)| a_{M_i}$ where $\text{card}(M_i)$ denotes the cardinality of $M_i$. 
    Choose $x=\sum m_j|d_j\in X$. 
    On every $a_{M_i}\wedge d_j>0$ it holds $m_j\in M_i$, the position\footnote{Each $M_i$ is an ordered set of the form $\{n_1^{M_i},\ldots, n^{M_i}_{k_{M_i}}\}$ and $m_{j,M_i}$ is the index such that $m_j=n^{M_i}_{m_{j,M_i}}$.} of which in the ordered set $M_i$ is denoted by $m_{j,M_i}$.
    Define $f(x)=\sum m_{j,M_i}|a_{i,M_i}\wedge d_j$.  
    Then $f:X\to \set{1\leqslant l\leqslant n}$ is stable by Proposition \ref{rule02}.  
    By construction, $n\in N$ is unique and $\mathbf{f}:\mathbf{X}\to \set{\mathbf{1}\leqslant \mathbf{l} \leqslant \mathbf{n}}$ a conditional bijection.
\end{proof}
\begin{definition}
    A conditional set $\mathbf{X}$ is \emph{conditionally countable} if there exists a conditional injection $\mathbf{f}:\mathbf{X}\to\mathbf{N}$.
    It is \emph{conditionally finite} if there exists a conditional bijection $\mathbf{f}:\mathbf{X}\to\set{\mathbf{1}\leqslant \mathbf{l}\leqslant \mathbf{n}}$ for some conditional element $\mathbf{n}$ of $\mathbf{N}$. 
\end{definition}
\begin{example}
    The conditional rational numbers $\mathbf{Q}$ are conditionally countable since every injection $f:\mathbb{Q}\to\mathbb{N}$ generates a conditional injection $\mathbf{f}:\mathbf{Q}\to\mathbf{N}$ by Example \ref{exep:Qfunc} 1). 
\end{example}
\begin{proposition}\label{lem:repfin}
    Let $\mathbf{X}$ be a conditional set and $(\mathbf{Y}_{\mathbf{n}})$ be a conditional sequence of conditional subsets of $\mathbf{X}$. 
    Then it holds:\footnote{$\sqcup_{\mathbf{1}\leqslant\mathbf{l}\leqslant \mathbf{n} } $ and $\sqcap_{\mathbf{1}\leqslant\mathbf{l}\leqslant \mathbf{n}}$ are understood as the conditional union and intersection over all conditional elements $\mathbf{l}$ such that $\mathbf{1}\leqslant \mathbf{l}\leqslant \mathbf{n}$.}
    \begin{enumerate}[fullwidth,label=(\roman*)]
        \item
            \begin{equation*} 
                \sqcup_{\mathbf{1}\leqslant \mathbf{l}\leqslant \mathbf{n}} \mathbf{Y}_{\mathbf{l}} =\sum \left(\sqcup_{l_i=1}^{n_i} \mathbf{Y}_{\mathbf{l}_{i}}\right)|a_i \quad \text{and} \quad 
                \sqcap_{\mathbf{1}\leqslant \mathbf{l}\leqslant \mathbf{n}} \mathbf{Y}_{\mathbf{l}} =\sum \left(\sqcap_{l_i=1}^{n_i} \mathbf{Y}_{\mathbf{l}_m}\right)|a_i\wedge b_i
            \end{equation*}
            where $n=\sum n_i|a_i$ and $\sqcap_{l_i=1}^{n_i} \mathbf{Y}_{\mathbf{l}_i}$ is a conditional subset on $b_i$ for each $i$. 
        \item\label{point3} If $\mathbf{Y}_{\mathbf{l}}$ is conditionally finite for each $\mathbf{1}\leqslant \mathbf{l}\leqslant \mathbf{n}$, then $\sqcup_{\mathbf{1}\leqslant \mathbf{l}\leqslant \mathbf{n}} \mathbf{Y}_{\mathbf{l}}$ is conditionally finite.
        \item \label{point4}
            If $\mathbf{Y}_{\mathbf{n}}$ is conditionally countable for each $\mathbf{n}$, then $\sqcup \mathbf{Y}_{\mathbf{n}}$ is conditionally countable. 
    \end{enumerate} 
\end{proposition}
\begin{proof}
    The statements are implied by Proposition \ref{rule02}.
\end{proof}
We close this section by a conditional axiom of choice. 
\begin{theorem}\label{thm:choice}
    Let $\mathbf{X}$ be a conditional set and $(\mathbf{Y}_{\mathbf{i}})$ be a conditional family of conditional subsets of $\mathbf{X}$.  
    Then there exists a conditional family $(\mathbf{y}_\mathbf{i})$ of conditional elements of $\mathbf{X}$ such that $\mathbf{y}_{\mathbf{i}}$ is a conditional element of $\mathbf{Y}_{\mathbf{i}}$  for each $\mathbf{i}$.  
\end{theorem}
\begin{proof}
    Let 
    \begin{equation*}
        \mathscr{H}:=\Set{(y_j)_{j\in J}: y_j\in Y_j \text{ for each } j\in J, J\in S(\mathbf{I})}. 
    \end{equation*}
    The set $\mathscr{H}$ is non-empty since it includes every one-element family.  
    Define an ordering on $\mathscr{H}$ by $$(y_j)_{j\in J} \leq (\bar{y}_j)_{j\in \bar{J}}\text{ whenever }J\subseteq \bar{J}\text{ and }y_j=\bar{y}_j \text{ for all }j\in J.$$  
    Let $(y_j)_{j\in J_\alpha}$ be a chain in $\mathscr{H}$, and put $$J=\Set{\sum j_\beta|a_\beta: (a_\beta)\in p(1), j_\beta\in J_{\alpha_\beta} \text{ for each }\beta}.$$ 
    By Proposition \ref{rule02}, one has $J\in S(\mathbf{I})$.  
    For $j=\sum j_\beta|a_\beta\in J$, define $y_j=\sum y_{j_\beta}|a_\beta$.  
    Since $(\mathbf{Y}_{\mathbf{i}})$ is a conditional family, it holds $(y_j)_{j\in J}\in \mathscr{H}$. 
    Inspection shows that $(y_j)_{j\in J_\alpha}\leq (y_j)_{j\in J}$ for each $\alpha$.   
    By Zorn's lemma, there exists a maximal element $(y_j)_{j\in J^\ast}$ in $\mathscr{H}$. 
    By way of contradiction, suppose there exists $i_0\in I$ such that $\mathbf{y}_{\mathbf{j}}\sqcap \mathbf{Y}_{\mathbf{i}_0}$ on some $b_j<1$ for all $j\in J^\ast$.   
    Let $\hat{J}=\set{j|a + i_0|a^c: a\in\mathcal{A},j\in J^\ast}$, pick some $y_{i_0}\in Y_{i_0}$ and define $y_j=y_j|a + y_{i_0}|a^c$ for each $j\in \hat{J}$. 
    Then $(y_j)_{j\in \hat{J}}$ is an element in $\mathscr{H}$. 
    However, one has $(y_j)_{j\in J^\ast}< (y_j)_{j\in \hat{J}}$ which is the desired contradiction. 
\end{proof}
\section{Conditional topology}\label{ch:topology}
Let $\mathbf{X}$ be a conditional set.
From the construction of the conditional power set it follows that 
\begin{itemize}
    \item $\mathbf{P}(\mathbf{P}(\mathbf{X}))$ is a conditional set of $P(\mathbf{P}(\mathbf{X}))$;
    \item $P(\mathbf{P}(\mathbf{X}))$ are the conditional subsets of $\mathbf{P}(\mathbf{X})$ on $1$;
    \item elements in $P(\mathbf{P}(\mathbf{X}))$ are generated by $S(\mathbf{P}(\mathbf{X}))$;
    \item $S(\mathbf{P}(\mathbf{X}))$ are stable subsets of $P(\mathbf{X})$.
\end{itemize}
Hence, an element in $\mathbf{P}(\mathbf{P}(\mathbf{X}))$ is a conditional collection of conditional subsets of $\mathbf{X}$ and an element in $S(\mathbf{P}(\mathbf{X}))$ is a stable collection of conditional subsets of $\mathbf{X}$ on $1$.
A stable collection $\mathcal{B}$ of conditional subsets of $\mathbf{X}$ on $1$ is in one-to-one relation to a stable collection $\mathscr{B}$ of stable subsets of $X$, the relation being given by $\mathcal{B}=\{ \mathbf{Y}\colon Y \in \mathscr{B} \}$ and $\mathscr{B}=\{ Y\colon \mathbf{Y}\in \mathcal{B} \}$.
\begin{definition}
    Let $\mathbf{X}$ be a conditional set and $\mathcal{T}$ a conditional collection of conditional subsets of $\mathbf{X}$.
    $\mathcal{T}$ is called a \emph{conditional topology} on $\mathbf{X}$ whenever
    \begin{enumerate}[label=(\roman*)]
        \item $\mathbf{X}\in \mathcal{T}$,
        \item if $\mathbf{O}_1,\mathbf{O}_2\in \mathcal{T}$, then $\mathbf{O}_1\sqcap \mathbf{O}_2\in\mathcal{T}$,
        \item if $(\mathbf{O}_i)$ is a non-empty collection in $\mathcal{T}$, then $\sqcup \mathbf{O}_i\in \mathcal{T}$.  
    \end{enumerate} 
    The pair $(\mathbf{X},\mathcal{T})$ is called a \emph{conditional topological space}. 
    A conditional set $\mathbf{O}\in \mathcal{T}$ is called \emph{conditionally open}. 
    A conditional subset $\mathbf{F}$ of $\mathbf{X}$ is called \emph{conditionally closed} whenever $\mathbf{F}^\sqsubset\in \mathcal{T}$.\footnote{
    Due to the duality principle in Boolean algebras, see \cite[p.~13]{monk1989handbook}, the conditional collection of all conditionally closed sets satisfies the dual properties of the conditional collection of all conditionally open sets.
    }
    Given two conditional topologies $\mathcal{T}_1$ and $\mathcal{T}_2$, $\mathcal{T}_1$ is said to be \emph{conditionally weaker} than $\mathcal{T}_2$ whenever $\mathcal{T}_1 \sqsubseteq \mathcal{T}_2$. 
    A conditional collection $\mathcal{B}$ of conditional subsets of $\mathbf{X}$ is a \emph{conditional topological base} whenever 
    \begin{enumerate}[label=\textit{(\roman*)}]
        \item\label{base:01} $\sqcup \mathcal{B}=\mathbf{X}$, 
        \item\label{base:02} if $\mathbf{O}_1,\mathbf{O}_2 \in \mathcal{B}$ and $\mathbf{x}$ is a conditional element of $\mathbf{O}_1\sqcap \mathbf{O}_2$, then there exists $\mathbf{O}_3\in \mathcal{B}$ such that $\mathbf{x}$ is a conditional element of $\mathbf{O}_3$ and $\mathbf{O}_3\sqsubseteq \mathbf{O}_1\sqcap \mathbf{O}_2$.  
    \end{enumerate} 
\end{definition}
The conditional topology \emph{conditionally generated} by a conditional collection $\mathcal{G}$ of conditional subsets of $\mathbf{X}$ is 
\begin{equation*}
    \mathcal{T}^{\mathcal{G}}:=\sqcap\Set{\mathcal{T}\colon \mathcal{T}\text{ conditional topology, }\mathcal{G}\sqsubseteq \mathcal{T}}. 
\end{equation*}
For a conditional topological base $\mathcal{B}$, inspection shows 
\begin{equation*}\label{eq:base}
\mathcal{T}^{\mathcal{B}}=\{\sqcup \mathbf{O_i}: (\mathbf{O_i}) \text{ non-empty collection in } \mathcal{B}\}. 
\end{equation*}
\begin{example}\label{exep:Qtop}
    For conditional elements $\mathbf{q},\mathbf{r}$ of $\mathbf{Q}$ such that $\mathbf{r>0}$, define the conditional set\footnote{See Example \ref{exep:Qfunc} 4) for the definition of the conditional absolute value.}
    \begin{equation*}
        \mathbf{B}_{\mathbf{r}}(\mathbf{q}):=\Set{\mathbf{p}\colon\abs{\mathbf{q}-\mathbf{p}}<\mathbf{r}}.
    \end{equation*}
    The conditional collection $\mathcal{B}$ of conditional sets generated by the stable collection
    \begin{equation*}
        \Set{\mathbf{B}_{\mathbf{r}}(\mathbf{q})\colon \mathbf{q},\mathbf{r}\text{ of } \mathbf{Q} \text{ with } \mathbf{r>0}}  
    \end{equation*}
    is a conditional topological base of $\mathcal{T}^\mathcal{B}$ called the \emph{conditional Euclidean topology} on $\mathbf{Q}$.
\end{example}
\begin{definition}
    Given a conditional topological space $(\mathbf{X},\mathcal{T})$ and a conditional subset $\mathbf{Y}$ of $\mathbf{X}$, the \emph{conditional interior} of $\mathbf{Y}$ is defined by 
    \begin{equation*}
        \Int(\mathbf{Y}):=\sqcup\Set{ \mathbf{O}\colon \mathbf{O}\text{ conditionally open, } \mathbf{O}\sqsubseteq \mathbf{Y}}, 
    \end{equation*}
    and its \emph{conditional closure} by
    \begin{equation*}
        \Cl(\mathbf{Y}):=\sqcap \set{\mathbf{F}\colon \mathbf{F} \text{ conditionally closed, }\mathbf{Y}\sqsubseteq \mathbf{F}}. 
    \end{equation*}
\end{definition} 
By the duality principle, one has
\begin{equation*}
    \Cl(\mathbf{Y})^\sqsubset=\Int\left(\mathbf{Y}^\sqsubset\right) \quad \text{and}\quad \Int(\mathbf{Y})^\sqsubset=\Cl\left(\mathbf{Y}^\sqsubset\right). 
\end{equation*} 
\begin{definition}
    Let $(\mathbf{X},\mathcal{T})$ be a conditional topological space and $\mathbf{x}$ a conditional element of $\mathbf{X}$. 
    A conditional subset $\mathbf{U}$ of $\mathbf{X}$ is a \emph{conditional neighborhood} of $\mathbf{x}$, if there exists a conditionally open set $\mathbf{O}$ such that $\mathbf{x}$ is a conditional element of $\mathbf{O}$ and $\mathbf{O} \sqsubseteq \mathbf{U}$.\footnote{In particular, $\mathbf{U}$ is on $1$.}   
    Let
    \begin{equation*}
        \mathcal{U}(\mathbf{x})=\set{\mathbf{U}\colon \mathbf{U} \text{ conditional neighborhood of } \mathbf{x}}
    \end{equation*}
    denote the stable collection of all conditional neighborhoods of $\mathbf{x}$.
    A \emph{conditional neighborhood base} of $\mathbf{x}$ is a stable collection $\mathcal{V}$ of conditional subsets of $\mathbf{X}$ on $1$ such that for every conditional neighborhood $\mathbf{U}$ of $\mathbf{x}$ there exists $\mathbf{V}\in \mathcal{V}$ such that $\mathbf{x}$ is a conditional element of $\mathbf{V}$ and $\mathbf{V} \sqsubseteq \mathbf{U}$.     
\end{definition}
A conditional topological space $(\mathbf{X},\mathcal{T})$ is \emph{conditionally first countable} if every conditional element $\mathbf{x}$ of $\mathbf{X}$ has a conditionally countable neighborhood base. 
It is \emph{conditionally second countable} if $\mathcal{T}$ is conditionally generated by a conditionally countable topological base. 
It is \emph{conditionally Hausdorff} if for every pair $\mathbf{x},\mathbf{y}$ of conditional elements of $\mathbf{X}$ with $\mathbf{x}\sqcap \mathbf{y}=\mathbf{X}|0$ there exists a conditional neighborhood $\mathbf{U}$ of $\mathbf{x}$ and a conditional neighborhood $\mathbf{V}$ of $\mathbf{y}$ such that $\mathbf{U}\sqcap \mathbf{V}=\mathbf{X}|0$.  
A conditional subset $\mathbf{Y}$ of $\mathbf{X}$ is \emph{conditionally dense} if $\Cl(\mathbf{Y})=\mathbf{X}$, and $(\mathbf{X},\mathcal{T})$ is \emph{conditionally separable} if $\mathbf{X}$ has a conditionally countable dense subset. 

Let $\mathscr{B}$ be a classical topological base on $X$. We denote by $\mathscr{T}^{\mathscr{B}}$ the classical topology generated by $\mathscr{B}$. Furthermore, denote by $\text{cl}(Y)$ the closure and by $\text{int}(Y)$ the interior of some $Y\subseteq X$ with respect to $\mathscr{T}^{\mathscr{B}}$. 
\begin{proposition}\label{prop:inducedtopology}
    Let $\mathbf{X}$ be a conditional set, $\mathscr{B}$ a stable collection of stable subsets of $X$ and $\mathcal{B}$ the corresponding conditional collection of conditional subsets of $\mathbf{X}$.
    Then $\mathcal{B}$ is a conditional topological base on $\mathbf{X}$ if, and only if, $\mathscr{B}$ is a classical topological base on $X$. 
    Moreover, it holds $$\set{O\in \mathscr{T}^{\mathscr{B}}\colon O\in S(\mathbf{X})}=\set{O\in S(\mathbf{X})\colon \mathbf{O}\in \mathcal{T}^{\mathcal{B}}}.$$   
\end{proposition} 
\begin{proof} 
    First assume that $\mathcal{B}$ is a conditional topological base.
    By Lemma \ref{lem:1fcunion}, one has $\cup \mathscr{B}=X$. 
    Let $O_1,O_2\in\mathscr{B}$ and $x\in O_1\cap O_2$. 
    By the definition of conditional intersection, $\mathbf{x}$ is a conditional element of $\mathbf{O}_1\sqcap \mathbf{O}_2$. 
    Hence there exists $\mathbf{O}_3\in\mathcal{B}$ such that $\mathbf{x}$ is a conditional element of $\mathbf{O}_3$ and $\mathbf{O}_3\sqsubseteq \mathbf{O}_1\sqcap \mathbf{O}_2$. 
    By Lemma \ref{lem:1fcunion}, one concludes $x\in O_3$ and $O_3 \subseteq O_1\cap O_2$.  
    Second assume that $\mathscr{B}$ is a classical base. 
    By Lemma \ref{lem:1fcunion}, it holds $\sqcup \mathcal{B}=\mathbf{X}$. 
    Let $\mathbf{O}_1,\mathbf{O}_2\in \mathcal{B}$, suppose $\mathbf{O}_1\sqcap \mathbf{O}_2$ is a conditional subset of $\mathbf{X}$ on $b$, and let $\mathbf{x}|b$ be a conditional element of $\mathbf{O}_1\sqcap \mathbf{O}_2$. 
    Since $\cup \mathscr{B}=X$, there exists $O\in \mathscr{B}$ such that $x\in O$. 
    Since $\mathscr{B}$ is stable, it holds $O_1|b + O|b^c, O_2|b + O|b^c\in \mathscr{B}$. 
    Moreover, $x\in (O_1|b + O|b^c)\cap (O_2|b + O|b^c)$.  
    Hence there exists $O_3\in\mathscr{B}$ such that $x\in O_3$ and $O_3\subseteq (O_1|b + O|b^c)\cap (O_2|b + O|b^c)$. 
    By \ref{cond2} and Lemma \ref{lem:1fcunion}, $\mathbf{x}|b$ is a conditional element of $\mathbf{O}_3|b$ and $\mathbf{O}_3|b\sqsubseteq \mathbf{O}_1\sqcap \mathbf{O}_2$. 
\end{proof} 
\begin{example}\label{e:l0top}
    Let $L^0_{++}:=\set{x\in L^0: X>0}$.
    Then for every $x\in L^0$ and each $r\in L^0_{++}$, 
    \begin{equation*}
        B_{r}(x):=\Set{y\in L^0: |X-Y|< R}
    \end{equation*}
    is a stable subset of $L^0$.
    The stable collection
    \begin{equation*}
        \mathscr{B}:=\Set{B_{r}(x): x\in L^0, r\in L^0_{++}}
    \end{equation*}
    generates the $L^0$-topology introduced in \cite{kupper03}.
    According to Proposition \ref{prop:inducedtopology}, the corresponding conditional collection $\mathcal{B}$ of conditional subsets of $\mathbf{L}^0$ is a conditional topological base generating $\mathcal{T}^{\mathcal{B}}$ which is conditionally Hausdorff and separable, see \cite[Lemma 5.3.2]{jamneshan13}.
\end{example}
\begin{proposition}\label{prop:cloint}
    Let $(\mathbf{X},\mathcal{T})$ be a conditional topological space and $\mathbf{Y}$ a conditional subset of $\mathbf{X}$ on $1$.  
    Then it holds
    \begin{align*}
        \Int(\mathbf{Y}) & = \sqcup \left\{ \mathbf{x}|b\colon \mathbf{U}|b\sqsubseteq \mathbf{Y} \text{ for some conditional neighborhood } \mathbf{U}\text{ of }\mathbf{x}\right\},\\
        \Cl(\mathbf{Y}) &= \sqcup \left\{ \mathbf{x}\colon \mathbf{U}\sqcap \mathbf{Y}\text{ is on }1\text{ for every conditional neighborhood }\mathbf{U}\text{ of }\mathbf{x} \right\}
    \end{align*}
    where $b=\vee\{a\colon \mathbf{O}\sqsubseteq \mathbf{Y} \text{ for some conditional open set }\mathbf{O}\text{ on }a\}$.
\end{proposition}
\begin{proof}
    The first assertion is immediate from the definitions. 
    As for the second one, assume that $\mathbf{x}$ is a conditional element of $\mathbf{F}$ for all conditionally closed sets $\mathbf{F}$ with $\mathbf{Y}\sqsubseteq \mathbf{F}$. 
    Suppose, for the sake of contradiction, that there exists a conditional open neighborhood $\mathbf{O}$ of $\mathbf{x}$ such that $\mathbf{O}\sqcap \mathbf{Y}$ is on $a<1$. 
    Then $\mathbf{x}|a^c\sqsubseteq \mathbf{Y}^\sqsubset|a^c$. 
    Since $\mathbf{O}^\sqsubset|a^c$ is conditionally closed, $\mathbf{O}^\sqsubset|a^c+\mathbf{X}|a$ is conditionally closed with $\mathbf{Y}\sqsubseteq \mathbf{O}^\sqsubset|a^c +\mathbf{X}|a$, which is the desired contradiction.   
    Conversely, assume that $\mathbf{x}$ is such that $\mathbf{U}\sqcap \mathbf{Y}$ is on $1$ for every conditional neighborhood $\mathbf{U}$ of $\mathbf{x}$.  
    Suppose, for the sake of contradiction, that there exists a conditionally closed set $\mathbf{F}$ with $\mathbf{Y}\sqsubseteq \mathbf{F}$ such that $\mathbf{x}\sqcap \mathbf{F}=\mathbf{x}|c$ for some $c<1$.  
    Then $\mathbf{x}|c^c\sqsubseteq \mathbf{F}^\sqsubset|c^c$. 
    Since $\mathbf{F}^\sqsubset|c^c$ is a conditionally open neighborhood of $\mathbf{x}|c^c$, it follows that $\mathbf{F}^\sqsubset|c^c +\mathbf{X}|c$ is a conditionally open neighborhood of $\mathbf{x}$. 
    By assumption, $(\mathbf{F}^\sqsubset|c^c +\mathbf{X}|c)\sqcap \mathbf{Y}$ is on $1$.   
    However, this contradicts $\mathbf{Y}\sqsubseteq \mathbf{F}$. 
\end{proof}

\subsection{Conditional continuity}

\begin{definition} 
    Let $(\mathbf{X},\mathcal{T})$ and $(\mathbf{X}^\prime,\mathcal{T}^\prime)$ be conditional topological spaces.
    A conditional function $\mathbf{f}:\mathbf{X}\to\mathbf{X}^\prime$ is said to be \emph{conditionally continuous} at the conditional element $\mathbf{x}$ of $\mathbf{X}$ if $\mathbf{f}^{-1}(\mathbf{U})$ is a conditional neighborhood of $\mathbf{x}$ for all conditional neighborhoods $\mathbf{U}$ of $\mathbf{f}(\mathbf{x})$.
    If $\mathbf{f}$ is conditionally continuous at every conditional element $\mathbf{x}$ of $\mathbf{X}$, then $\mathbf{f}$ is said to be conditionally continuous.  
    Let $(\mathbf{X}_i,\mathcal{T}_i)_{i\in I}$ be a non-empty family of conditional topological spaces and $(\mathbf{f}_i)_{i\in I}$ be a family of conditional functions $\mathbf{f}_i:\mathbf{X}\to \mathbf{X}_i$. 
    The \emph{conditional initial topology} on $\mathbf{X}$ for the family $(\mathbf{f}_i)_{i\in I}$ is the conditional topology generated by $\cond(\mathcal{G})$ where $\mathcal{G}:=\set{\mathbf{f}_i^{-1}(\mathbf{O}_i):\mathbf{f}_i^{-1}(\mathbf{O}_i)\text{ on $1$}, \mathbf{O}_i \in \mathcal{T}_i, i\in I}$.\footnote{Note that if $\mathbf{f}_i^{-1}(\mathbf{O}_i)$ is on $a<1$ for some $\mathbf{O}_i\in\mathcal{T}_i$, then $\mathbf{f}_i^{-1}(\mathbf{O}_i|a+\mathbf{X}_i|a^c)$ is on $1$.}
\end{definition}
\begin{examples}\label{exp:inittop}
    \begin{enumerate}[label=\arabic*), fullwidth]
        \item Let $(\mathbf{X},\mathcal{T})$ be a conditional topological space, $\mathbf{Y}$ a conditional subset of $\mathbf{X}$ on $1$ and $\mathbf{f}:\mathbf{Y}\to \mathbf{X}$ a conditional embedding, see Example \ref{exep:Qfunc} 3). 
            The \emph{conditional relative topology} of $\mathcal{T}$ with respect to $\mathbf{Y}$ is the conditional initial topology for $\mathbf{f}$. 
        \item Let $(\mathbf{X}_i,\mathcal{T}_i)$ be a non-empty family of \emph{conditional topological spaces}. 
            The conditional product topology on $\prod \mathbf{X}_i$ is the conditional initial topology for the family of conditional projections $(\mathbf{pr}_i)$.
    \end{enumerate}
\end{examples}
\begin{proposition}\label{p:cts}
    Let $(\mathbf{X},\mathcal{T})$ and $(\mathbf{X}^\prime,\mathcal{T}^\prime)$ be conditional topological spaces and $\mathbf{f}:\mathbf{X}\to\mathbf{X}^\prime$ a conditional function. 
    The following are equivalent:
    \begin{enumerate}[label=\textit{(\roman*)}]
        \item\label{cond:cont01} The conditional function $\mathbf{f}$ is conditionally continuous. 
        \item $\mathbf{f}^{-1}(\mathbf{O}^\prime)$ is conditionally open for every conditionally open set $\mathbf{O}^\prime$.
        \item\label{cond:cont02} $\mathbf{f}^{-1}(\mathbf{F}^\prime)$ is conditionally closed for every conditionally closed set $\mathbf{F}^\prime$. 
    \end{enumerate}
\end{proposition}
\begin{proof}
    The assertions are immediate from the definitions.
\end{proof}
\begin{proposition}\label{p:obenuntenstetig}
    Let $\mathbf{X}$ and $\mathbf{X}^\prime$ be two conditional sets, $\mathscr{B}$ and $\mathscr{B}^\prime$  stable collections of stable sets which are a base of a topology on $X$ and $X^\prime$, respectively, and $\mathcal{B}$ and $\mathcal{B}^\prime$  the corresponding conditional topological bases on $\mathbf{X}$ and $\mathbf{X}^\prime$, respectively.
    A conditional function $\mathbf{f}:\mathbf{X}\to \mathbf{X}^\prime$ is conditionally continuous if, and only if, $f:X\to X^\prime$ is continuous with respect to the topologies $\mathscr{T}^{\mathscr{B}}$ and $\mathscr{T}^{\mathscr{B}^\prime}$. 
\end{proposition}
\begin{proof}
    Assume that $f:X\to X^\prime$ is continuous.
    Let $\mathbf{O}\in\mathcal{B}^\prime$, and suppose that $\mathbf{f}^{-1}(\mathbf{O})$ is a conditional subset of $\mathbf{X}$ on $a$. 
    By Proposition \ref{prop:inducedtopology}, one has
    \begin{equation*}
        f^{-1}(O)|a+ X|a^c=f^{-1}(O|a+X^\prime|a^c)\in \set{O\in\mathscr{T}^{\mathscr{B}}:O\in S(\mathbf{X})}=\set{O\in S(\mathbf{X}): \mathbf{O}\in \mathcal{T}^{\mathcal{B}}}.
    \end{equation*}
    Thus $\mathbf{f}^{-1}(\mathbf{O}) = \mathbf{f}^{-1}(\mathbf{O}|a+\mathbf{X}^\prime|{a^c})|a \in \mathcal{T}^{\mathcal{B}}$. 
    Conversely, assume that $\mathbf{f}:\mathbf{X}\to\mathbf{X}^\prime$ is conditionally continuous and let $O\in\mathscr{B}^\prime$.
    Without loss of generality, we may assume that $f^{-1}(O)\neq\emptyset$. By Proposition \ref{prop:inducedtopology}, one has $f^{-1}(O)\in \set{O\in S(\mathbf{X}): \mathbf{O}\in \mathcal{T}^{\mathcal{B}}}=\set{O\in\mathscr{T}^{\mathscr{B}}:O\in S(\mathbf{X})}$, and therefore $f^{-1}(O)\in \mathscr{T}^{\mathscr{B}}$.
\end{proof} 
\subsection{Conditional filters}
\begin{definition}
    Let $\mathbf{X}$ be a conditional set.  
    A \emph{conditional filter} $\mathcal{F}$ on $\mathbf{X}$ is a stable collection of conditional subsets of $\mathbf{X}$ on $1$ satisfying the following conditions: 
    \begin{enumerate}[label=(\roman*)]
        \item\label{cond:filter01} if $\mathbf{Y}\in \mathcal{F}$ and $\mathbf{Y}\sqsubseteq \mathbf{Z}\sqsubseteq \mathbf{X}$, then $\mathbf{Z}\in \mathcal{F}$;
        \item\label{cond:filter02} if $\mathbf{Y}, \mathbf{Z}\in \mathcal{F}$, then $\mathbf{Y} \sqcap \mathbf{Z}\in \mathcal{F}$.  
    \end{enumerate}
    A conditional filter $\mathcal{F}$ is \emph{conditionally finer} than a conditional filter $\mathcal{F}^\prime$ if $\mathcal{F}^\prime\sqsubseteq \mathcal{F}$.\footnote{The conditional inclusion between the two stable collections of conditional subsets of $\mathbf{X}$ is understood in the sense of the conditional inclusion of the generated conditional sets.}
    A \emph{conditional ultrafilter} is a maximal element in the set of all conditional filters on $\mathbf{X}$. 
    A stable collection $\mathcal{B}$ of conditional subsets of $\mathbf{X}$ on $1$ is a \emph{conditional filter base} if for every $\mathbf{Y}_1,\mathbf{Y}_2\in \mathcal{B}$ there exists $\mathbf{Y}_3\in \mathcal{B}$ with $\mathbf{Y}_3  \sqsubseteq \mathbf{Y}_1 \sqcap \mathbf{Y}_2$.  
\end{definition}
\begin{remark}\label{rem:minimalknoten}
    By consistency, if $\mathcal{F}$ is a conditional filter on $\mathbf{X}$, then $\mathcal{F}|a$ is a conditional filter on $\mathbf{X}|a$.
    Moreover, let $\mathcal{F}:=\{\mathbf{Y}_i|a_i\colon i\}$ be a collection of conditional subsets of $\mathbf{X}$ not necessarily on $1$ such that $\sum \left(\mathbf{Y}_{i_j}|a_{i_j}\right)|b_j=\sum \mathbf{Y}_{i_j}|a_{i_j}\wedge b_j \in \mathcal{F}$ for every $(b_j)\in p(1)$ and each family $(\mathbf{Y}_{i_j}|a_{i_j})$ of elements in $\mathcal{F}$.
    Suppose that $\mathcal{F}$ satisfies \ref{cond:filter01} and \ref{cond:filter02} in the definition of a filter and additionally $\mathbf{X}|0 \not \in \mathcal{F}$ as in classical filter's definition.
    Then it follows that there exists a minimal condition $a^{\mathcal{F}}>0$ such that $\mathcal{F}|a^\mathcal{F}$ is a conditional filter on $\mathbf{X}|a^\mathcal{F}$.\footnote{
Indeed, suppose, for the sake of contradiction, that $\wedge a_i = 0$.
By de Morgan's law, it holds $\vee a_i^c=1$.
By the well-ordering theorem, choose $(b_i)\in p(1)$ such that $b_i\leq a_{i}^c$ for each $i$. Then $\sum \mathbf{Y}_{i}|b_i=\mathbf{X}|0$.
Since $\mathcal{F}$ is a stable collection of conditional sets, it follows that $\sum \mathbf{Y}_{i}|b_i \in \mathcal{F}$ which contradicts $\mathbf{X}|0\not\in\mathcal{F}$. 
    }
\end{remark}
For every conditional filter base $\mathcal{B}$, inspection shows that
\begin{equation*}
    \mathcal{F}^{\mathcal{B}}:=\Set{\mathbf{Z}\colon\mathbf{Y}\sqsubseteq \mathbf{Z}\sqsubseteq \mathbf{X} \text{ for some }\mathbf{Y}\in \mathcal{B}}   
\end{equation*}
is a conditional filter, called the conditional filter \emph{conditionally generated} by $\mathcal{B}$.  
\begin{examples}
    \begin{enumerate}[fullwidth,label=\arabic*)]
        \item The conditional trivial filter is $\set{\mathbf{X}}$. 
        \item The stable collection $\mathcal{U}(\mathbf{x})$ of all conditional neighborhoods of a conditional element $\mathbf{x}$ of a conditional topological space is a conditional filter. 
        \item For any conditional subset $\mathbf{Y}$ of $\mathbf{X}$ on $1$, the collection $\set{\mathbf{Z}\colon \mathbf{Y}\sqsubseteq \mathbf{Z}\sqsubseteq \mathbf{X}}$ is a conditional filter. 
    \end{enumerate}
\end{examples}
\begin{proposition}\label{prop:filtercondfilter}
    Let $\mathbf{X}$ be a conditional set and $\mathcal{B}$  a stable collection of conditional subsets of $\mathbf{X}$ on $1$.
    Then $\mathcal{B}$ is a conditional filter base if, and only if, the corresponding stable collection $\mathscr{B}$ of stable subsets of $X$ is a classical filter base on $X$. 
\end{proposition}
\begin{proof}
    By Lemma \ref{lem:1fcunion}, the equivalence is immediate from the definitions. 
\end{proof}
We next prove a conditional ultrafilter lemma. 
\begin{theorem}\label{bpit}
    For every conditional filter $\mathcal{F}$ there exists a conditional ultrafilter $\mathcal{U}$ such that $\mathcal{F}\sqsubseteq \mathcal{U}$.
\end{theorem}
\begin{proof}
    Order $\mathscr{F}:=\Set{\mathcal{G} \colon \mathcal{G} \text{ is a conditional filter with } \mathcal{F}\sqsubseteq \mathcal{G}}$ by conditional inclusion. 
    Let $(\mathcal{G}_i)$ be a chain in $\mathscr{F}$ and set $\mathcal{W}:=\sqcup \mathcal{G}_i$.
    We show that $\mathcal{W}$ is a conditional filter with $\mathcal{F}\sqsubseteq \mathcal{W}$.
    Let $\mathbf{Z}\sqsubseteq \mathbf{X}$ be such that $\mathbf{Y}\sqsubseteq \mathbf{Z}$ for some $\mathbf{Y}\in \mathcal{W}$. 
    Then $\mathbf{Y}=\sum \mathbf{Y}_j|a_j$ for some $(a_j)\in p(1)$ and $\mathbf{Y}_j\in \mathcal{G}_{i_j}$ for each $j$. 
    For each $j$, it holds $\mathbf{Z}|a_j\in \mathcal{G}_{i_j}|a_j$ because $\mathbf{Y}_j|a_j\sqsubseteq \mathbf{Z}|a_j$ and $\mathcal{G}_{i_j}|a_j$ is a conditional filter on $\mathbf{X}|a_j$.
    By stability,  $\mathbf{Z}|a_j+\mathbf{X}|a_j^c\in \mathcal{W}$ for all $j$, and thus $\mathbf{Z}=\sum (\mathbf{Z}|a_j+\mathbf{X}|a_j^c)|a_j\in \mathcal{W}$. 
    Let $\mathbf{Y},\mathbf{Z}\in \mathcal{W}$ where $\mathbf{Y}=\sum \mathbf{Y}_j|a_j$ and $\mathbf{Z}=\sum \mathbf{Z}_l|b_l$ for some $(a_j),(b_l)\in p(1)$, and $\mathbf{Y}_j\in \mathcal{G}_{i_j}$ for each $j$ and $\mathbf{Z}_l\in \mathcal{G}_{i_l}$ for each $l$. 
    Since $(\mathcal{G}_i)$ is a chain, it holds $\mathbf{Y}_j\sqcap \mathbf{Z}_l$ is a conditional subset on $1$ in $\mathcal{G}_i$ for $i=j\vee l$ for every $j,l$. 
    By stability, it holds that $\mathbf{Y}\sqcap \mathbf{Z}=\sum (\mathbf{Y}_j\sqcap \mathbf{Z}_l)|a_j\wedge b_l$ is a conditional set on $1$ in $\mathcal{W}$.
    Hence $\mathcal{W}$ is a conditional filter with $\mathcal{F}\sqsubseteq \mathcal{W}$ and therefore an upper bound for $(\mathcal{G}_i)$.
    The existence of a conditional ultrafilter follows by Zorn's lemma.
\end{proof}
\begin{proposition}\label{lem:ufilterprop}
    Let $\mathcal{U}$ be a conditional filter. 
    Then the following are equivalent:
    \begin{enumerate}[label=(\textit{\roman*})]
        \item\label{lem:001} $\mathcal{U}$ is a conditional ultrafilter. 
        \item\label{lem:003}  If $\mathbf{Y}_1\sqcup \mathbf{Y}_2\in \mathcal{U}$ for some $\mathbf{Y}_1,\mathbf{Y}_2\sqsubseteq \mathbf{X}$, then $\mathbf{Y}_1|a+\mathbf{Y}_2|a^c \in \mathcal{U}$ where either $a=a_1$ or $a^c=a_2$, whereby $\mathbf{Y}_1$ is on $a_1$ and $\mathbf{Y}_2$ is on $a_2$. 
        \item\label{lem:002} For every $\mathbf{Y}\sqsubseteq \mathbf{X}$, it holds $\mathbf{Y}|a+\mathbf{Y}^{\sqsubset}|a^c\in \mathcal{U}$ where either $a=a_1$ or $a^c=a_2$, whereby $\mathbf{Y}$ is on $a_1$ and $\mathbf{Y}^\sqsubset$ is on $a_2$. 
        \item\label{lem:004} For every $\mathbf{Y}\sqsubseteq \mathbf{X}$ such that $\mathbf{Y}\sqcap \mathbf{U}$ is on $1$ for every $\mathbf{U}\in \mathcal{U}$, it holds $\mathbf{Y} \in \mathcal{U}$. 
    \end{enumerate}
\end{proposition}
\begin{proof}
    \begin{enumerate}[fullwidth]
        \item[] To show that \ref{lem:001} implies \ref{lem:003}, let $\mathbf{Y}_1,\mathbf{Y}_2\sqsubseteq \mathbf{X}$ be such that $\mathbf{Y}_1\sqcup \mathbf{Y}_2\in \mathcal{U}$. 
            Since
            \begin{equation*}\label{rep2union}
                \mathbf{Y}_1\sqcup \mathbf{Y}_2=\mathbf{Y}_1|b_1 + (\mathbf{Y}_1\sqcup \mathbf{Y}_2)|b_2+\mathbf{Y}_2|b_3
            \end{equation*}
            where $b_1=a_1\wedge a^c_2$, $b_2=a_1\wedge a_2$ and $b_3=a_2\wedge a^c_1$, it holds $\mathbf{Y}_1|b_1, \mathbf{Y}_1|b_2\sqcup \mathbf{Y}_2|b_2, \mathbf{Y}_2|b_3 \in \mathcal{U}$. 
            Now if $\mathbf{Y}_1|b_2 \in \mathcal{U}$, then $a=b_1\vee b_2=a_1$ yields the claim.  
            Otherwise $\mathcal{F}:=\set{\mathbf{Z}\sqsubseteq \mathbf{X} \colon \mathbf{Z}\sqcup \mathbf{Y}_1|b_2\in \mathcal{U}}$ is a conditional filter such that $\mathbf{Y}_2|b_2\in \mathcal{F}$.  
            Since $\mathcal{U}\sqsubseteq \mathcal{F}$ and $\mathcal{U}$ is a conditional ultrafilter, it holds  $\mathbf{Y}_2|b_2 \in\mathcal{U}$.  
            In that case $a^c=b_2\vee b_3=a_2$ yields the assertion. 
        \item[]  As for \ref{lem:003} implies \ref{lem:002}, set $\mathbf{Y}_1:=\mathbf{Y}$ and $\mathbf{Y}_2:=\mathbf{Y}^\sqsubset$. Then \ref{lem:002} follows since $\mathbf{Y}\sqcup \mathbf{Y}^\sqsubset =\mathbf{X} \in \mathcal{U}$.   
        \item[] To show that \ref{lem:002} implies \ref{lem:001}, let $\mathcal{V}$ be a conditional filter conditionally finer than $\mathcal{U}$ and $\mathbf{Y}\in\mathcal{V}$. 
            For the sake of contradiction, suppose $\mathbf{Y} \not\in \mathcal{U}$. 
            Then by assumption $\mathbf{Y}|a+\mathbf{Y}^{\sqsubset}|a^c\in \mathcal{U}$ where $\mathbf{Y}^{\sqsubset}$ is on $a^c>0$. 
            Hence $\mathbf{Y}^{\sqsubset}|a^c\in \mathcal{U}\sqsubseteq \mathcal{V}$.
            However, since $\mathbf{Y}|a^c +\mathbf{X}|a,\mathbf{Y}^\sqsubset|a^c+\mathbf{X}|a\in \mathcal{V}$, it holds $(\mathbf{Y}|a^c +\mathbf{X}|a)\sqcap(\mathbf{Y}^\sqsubset|a^c+\mathbf{X}|a)=\mathbf{X}|a$ which contradicts the second property of a conditional filter. 
            Thus $a^c=0$, showing that $\mathcal{V}= \mathcal{U}$. 
        \item[] Finally we show that \ref{lem:001} is equivalent to \ref{lem:004}.
            Assume \ref{lem:001} and let $\mathbf{Y}\sqsubseteq \mathbf{X}$ be such that $\mathbf{Y}\sqcap \mathbf{U}$ is on $1$ for every $\mathbf{U}\in\mathcal{U}$. 
            Inspection shows that $\mathcal{B}:=\set{\mathbf{Y}\sqcap \mathbf{U}:\mathbf{U}\in\mathcal{U}}$ is a conditional filter base with $\mathcal{U}\sqsubseteq\mathcal{F}^{\mathcal{B}}$.
            Hence $\mathcal{U}=\mathcal{F}^\mathcal{B}$, and thus $\mathbf{Y} \in \mathcal{U}$.
            Conversely, let $\mathcal{V}$ be a conditional ultrafilter of $\mathcal{U}$ and let $\mathbf{Y} \in\mathcal{V}$. 
            From $\mathcal{U}\sqsubseteq \mathcal{V}$ it follows that $\mathbf{Y}\sqcap \mathbf{U}$ is on $1$ for every $\mathbf{U} \in \mathcal{U}$. 
            By assumption, one has $\mathbf{Y}\in \mathcal{U}$, and therefore $\mathcal{V}= \mathcal{U}$. 
    \end{enumerate}
\end{proof}
\begin{proposition}\label{prop:ultrafilterfonct}
    Let $\mathbf{X}$ and $\mathbf{Y}$ be conditional sets, $\mathbf{f}:\mathbf{X}\to\mathbf{Y}$ a conditional function, $\mathcal{F}$  a conditional filter on $\mathbf{X}$ and $\mathcal{U}$  a conditional ultrafilter on $\mathbf{X}$.  
    Then $\mathbf{f}(\mathcal{F}):=\set{\mathbf{f}(\mathbf{U}):\mathbf{U}\in\mathcal{F}}$ is a conditional filter base on $\mathbf{Y}$ and $\mathbf{f}(\mathcal{U})$ a conditional ultrafilter base on $\mathbf{Y}$. 
\end{proposition}
\begin{proof}
    Since $\mathcal{F}$ is a stable collection of conditional subsets of $\mathbf{X}$, it follows from the stability of $f$ that $\mathbf{f}(\mathcal{F})$ is also a stable collection of conditional subsets of $\mathbf{X}$.  
    It is immediate from the definitions, \eqref{f:021} and \eqref{f:041} that $\mathbf{f}(\mathbf{F})$ is a conditional filter base. 
    Next suppose that $\mathcal{U}$ is a conditional ultrafilter and denote by $\mathcal{V}$ the conditional filter conditionally generated by $\mathbf{f}(\mathcal{U})$.
    Let $\mathbf{V}\sqsubseteq \mathbf{Y}$ be on $a_1$ and $\mathbf{V}^\sqsubset$ on $a_2$. 
    By \eqref{f:031} and \eqref{f:051},  
    \begin{equation*}
        \mathbf{f}^{-1}(\mathbf{V})\sqcup \mathbf{f}^{-1}(\mathbf{V})^\sqsubset=\mathbf{f}^{-1}(\mathbf{V}\sqcup \mathbf{V}^{\sqsubset})=\mathbf{f}^{-1}(\mathbf{Y})=\mathbf{X}\in \mathcal{U}.
    \end{equation*}
    Since $\mathcal{U}$ is a conditional ultrafilter, Proposition \ref{lem:ufilterprop} implies that $\mathbf{U}:=\mathbf{f}^{-1}(\mathbf{V})|b+\mathbf{f}^{-1}(\mathbf{V}^{\sqsubset})|b^c \in \mathcal{U}$ where either $b=b_1$ or $b^c=b_2$ whereby $\mathbf{f}^{-1}(\mathbf{V})$ is on $b_1\leq a_1$ and $\mathbf{f}^{-1}(\mathbf{V}^\sqsubset)$ is on $b_2\leq a_2$.   
    By \eqref{f:051} and \eqref{f:061}, it holds 
    \begin{equation*}
        \mathbf{f}(\mathbf{U})=\mathbf{f}\left(\mathbf{f}^{-1}(\mathbf{V}) \right)|b+\mathbf{f}\left( \mathbf{f}^{-1}(\mathbf{V}^\sqsubset) \right)|b^c\sqsubseteq \mathbf{V}|b+\mathbf{V}^{\sqsubset}|b^c. 
    \end{equation*}
    Since $\mathbf{f}(\mathbf{U})\in\mathbf{f}(\mathcal{U})$, one has $\mathbf{V}|b+\mathbf{V}^{\sqsubset}|b^c \in \mathcal{V}$. 
    Without loss of generality, assume that $b=b_1\leq a_1$. 
    By concatenating, we obtain $\mathbf{V}|a+\mathbf{V}^{\sqsubset}|a^c \in \mathcal{V}$ where $a=a_1$.  Proposition \ref{lem:ufilterprop} implies that $\mathcal{V}$ is a conditional ultrafilter.
\end{proof}
\subsection{Conditional convergence}
\begin{definition}
    Let $(\mathbf{X},\mathcal{T})$ be a conditional topological space, $\mathcal{F}$  a conditional  filter and $(\mathbf{x}_{\mathbf{i}})$  a conditional net of conditional elements of $\mathbf{X}$.
    A conditional element $\mathbf{x}$ of $\mathbf{X}$ is said to be a 
    \begin{enumerate}[label=(\textit{\roman*})]
        \item \emph{conditional limit point} of $\mathcal{F}$ if $\mathcal{U}(\mathbf{x})\sqsubseteq \mathcal{F}$;
        \item \emph{conditional cluster point} of $\mathcal{F}$ if $\mathbf{x}$ is a conditional element of $\Cl(\mathbf{Y})$ for every $\mathbf{Y}\in \mathcal{F}$ and denote $\Lim \mathcal{F}:=\sqcap \set{\Cl(\mathbf{Y})\colon\mathbf{Y} \in \mathcal{F}}$;
        \item \emph{conditional limit point} of $(\mathbf{x}_{\mathbf{i}})$ if for every conditional neighborhood $\mathbf{U}$ of $\mathbf{x}$ there exists $\mathbf{i}_0$ such that each $\mathbf{x}_{\mathbf{i}}$ is a conditional element of $\mathbf{U}$ for every $\mathbf{i}\geqslant \mathbf{i}_0$;
        \item \emph{conditional cluster point} of $(\mathbf{x}_{\mathbf{i}})$ if for every conditional neighborhood $\mathbf{U}$ of $\mathbf{x}$ and every $\mathbf{i}$ there exists $\mathbf{j}\geqslant \mathbf{i}$ such that $\mathbf{x}_{\mathbf{j}}$ is a conditional element of $\mathbf{U}$.
    \end{enumerate}
\end{definition}
We indicate by $\mathcal{F}\to \mathbf{x}$ that $\mathbf{x}$ is the conditional limit point of $\mathcal{F}$.
For a classical topology $\mathscr{T}$ and a classical filter $\mathscr{F}$, denote 
by $\mathscr{F}\xrightarrow[]{\mathscr{T}} x$ the convergence of $\mathscr{F}$ to $x$ and 
by $\Limcstd \mathscr{F}$ the set of all cluster points of $\mathscr{F}$ with respect to $\mathscr{T}$.
\begin{remark}
    As in the classical case, inspection shows that conditional filters and conditional nets are in one-to-one relation.
    The following Propositions \ref{clustultra}, \ref{p:konvergenzobenunten} and \ref{stetigmitnetz} are formulated in terms of conditional filters.
    Their respective analogues hold for conditional nets.
\end{remark}
\begin{proposition}\label{clustultra}
    Let $(\mathbf{X},\mathcal{T})$ be a conditional topological space and $\mathcal{F}$  a conditional filter. 
    Then the following are equivalent:
    \begin{enumerate}[label=(\roman*)]
        \item\label{cond:filterconv01} The conditional element $\mathbf{x}$ is a conditional element of  $\Lim \mathcal{F}$. 
        \item\label{cond:filterconv03} There exists a conditional filter $\mathcal{G}$ conditionally finer than $\mathcal{F}$ such that $\mathcal{G}\to \mathbf{x}$.
    \end{enumerate} 
\end{proposition}
\begin{proof}
    To show that \ref{cond:filterconv01} implies \ref{cond:filterconv03}, let $\mathbf{x}$ be a limit point of $\mathcal{F}$.
    Then $\Set{\mathbf{V}\sqcap \mathbf{U}:\mathbf{V}\in \mathcal{U}(\mathbf{x}), \mathbf{U}\in\mathcal{F}}$ is a conditional filter base of a conditional filter $\mathcal{G}$ conditionally finer than $\mathcal{F}$ and for which holds $\mathcal{G}\to \mathbf{x}$. 
    To show that \ref{cond:filterconv03} implies \ref{cond:filterconv01}, let $\mathcal{G}$ be a conditional filter conditionally finer than $\mathcal{F}$ and $\mathcal{G}\to \mathbf{x}$.
    Then $\mathbf{V}\sqcap \mathbf{Y}$ is on $1$ for all $\mathbf{V}\in \mathcal{U}(\mathbf{x})$ and $\mathbf{Y}\in \mathcal{F}$ since $\mathbf{V},\mathbf{Y}\in \mathcal{G}$. 
    Proposition \ref{prop:cloint} implies that $\mathbf{x}$ is a conditional element of $\Cl(\mathbf{Y})$ for all $\mathbf{Y}\in \mathcal{F}$ which shows that $\mathbf{x}$ is a conditional limit point of $\mathcal{F}$.
\end{proof}
\begin{proposition}\label{p:konvergenzobenunten}
    Let $\mathbf{X}$ be a conditional set, $\mathcal{G}$  a conditional filter base, $\mathcal{B}$  a conditional topological base, and $\mathscr{G}$ and $\mathscr{B}$  the corresponding stable collections of stable subsets of $X$.  
    Then 
    \begin{align*}
    x \in \Limcstd\mathscr{G}    \quad &\text{ if, and only if, }\quad \mathbf{x}\text{ is a conditional element of }\Lim \mathcal{G}\\
       \mathscr{G}\xrightarrow[]{\mathscr{T}^{\mathscr{B}}}x \quad &\text{ if, and only if, }\quad\mathcal{G}\to \mathbf{x}.
    \end{align*}
\end{proposition}
\begin{proof}
    The assertions follow from the Propositions \ref{prop:inducedtopology}, \ref{prop:cloint} and \ref{prop:filtercondfilter}.
\end{proof}
\begin{proposition}\label{stetigmitnetz}
    Let $(\mathbf{X},\mathcal{T})$ and $(\mathbf{X}^\prime,\mathcal{T}^\prime)$ be conditional topological spaces and $\mathbf{f}:\mathbf{X}\to \mathbf{X}^\prime$  a conditional function.
    Then the following are equivalent: 
    \begin{enumerate}[label=\textit{(\roman*)}]
        \item\label{cont1} The conditional function $\mathbf{f}$ is conditionally continuous at $\mathbf{x}$. 
        \item For every conditional filter $\mathcal{F}\to \mathbf{x}$, it holds $\mathbf{f}( \mathcal{F})\to \mathbf{f}(\mathbf{x})$. 
    \end{enumerate}
\end{proposition}
\begin{proof}
    By Proposition \ref{p:obenuntenstetig}, conditional continuity is equivalent to continuity. 
    The claim follows from the respective classical result and Proposition \ref{p:konvergenzobenunten}.
\end{proof}
\subsection{Conditional compactness}
Let $(\mathbf{X},\mathcal{T})$ be a conditional topological space.
A \emph{conditional open covering} of $\mathbf{X}$ is a conditional family $(\mathbf{O}_\mathbf{i})$ of conditional open sets such that $\mathbf{X} = \sqcup \mathbf{O}_{\mathbf{i}}$.
A conditional family $(\mathbf{Y}_\mathbf{i})$ of conditional subsets of $\mathbf{X}$ has the \emph{conditional finite intersection property} if\footnote{Recall that $\sqcup_{\mathbf{1}\leqslant\mathbf{l}\leqslant \mathbf{n} } $ and $\sqcap_{\mathbf{1}\leqslant\mathbf{l}\leqslant \mathbf{n}}$ are understood as the conditional union and intersection over all conditional elements $\mathbf{l}$ such that $\mathbf{1}\leqslant \mathbf{l}\leqslant \mathbf{n}$.} $\sqcap_{\mathbf{1}\leqslant \mathbf{l} \leqslant \mathbf{n}} \mathbf{Y}_{\mathbf{i}_{\mathbf{l}}}$ is on $1$ for every conditionally finite subfamily $(\mathbf{Y}_{\mathbf{i}_{\mathbf{l}}})_{\mathbf{1}\leqslant \mathbf{l} \leqslant \mathbf{n}}$. 
\begin{definition}
    We say that $\mathbf{X}$ is \emph{conditionally compact} if for every conditional open covering $(\mathbf{O}_{\mathbf{i}})$ there exists a conditionally finite subfamily $(\mathbf{O}_{\mathbf{i}_\mathbf{l}})_{\mathbf{1}\leqslant \mathbf{l} \leqslant \mathbf{n}}$ such that $\mathbf{X}= \sqcup_{\mathbf{1}\leqslant \mathbf{l} \leqslant \mathbf{n}} \mathbf{O}_{\mathbf{i}_\mathbf{l}}$. 
\end{definition}
\begin{proposition}\label{characcomp}
    Let $(\mathbf{X},\mathcal{T})$ be a conditional topological space. 
    Then the following are equivalent:
    \begin{enumerate}[label=(\roman*)]
        \item\label{cond:comp01} $\mathbf{X}$ is conditionally compact.
        \item\label{cond:comp03} Every conditional filter on $\mathbf{X}$ has a conditional cluster point. 
        \item\label{cond:comp04} Every conditional ultrafilter on $\mathbf{X}$ has a conditional limit point. 
        \item\label{cond:comp05} For every conditional family $(\mathbf{F}_\mathbf{i})$ of conditional closed subsets of $\mathbf{X}$ with the conditional finite intersection property, $\sqcap \mathbf{F}_{\mathbf{i}}$ is on $1$. 
    \end{enumerate}
\end{proposition}
\begin{proof}
    \begin{enumerate}[fullwidth]
        \item[] The equivalence of \ref{cond:comp03} and \ref{cond:comp04} follows from Theorem \ref{bpit} and Proposition \ref{clustultra}. 
        \item[] To show that \ref{cond:comp01} implies \ref{cond:comp05}, let $(\mathbf{F}_\mathbf{i})$ be a conditional family of conditional closed subsets of $\mathbf{X}$ with the conditional finite intersection property. 
        By contradiction, assume that $\sqcap \mathbf{F}_\mathbf{i}$ is on $a<1$. 
        Let $\mathbf{O}_{\mathbf{i}}:=(\mathbf{F}_{\mathbf{i}})^{\sqsubset} \in \mathcal{T}$ for each $\mathbf{i}$.
        Without loss of generality, assume that $\mathbf{O_i}$ is on $1$ for each $\mathbf{i}$. Otherwise, replace $\mathbf{O_i}$ by $\mathbf{\hat{O}_i}=\mathbf{O_i}|a_\mathbf{i}+\mathbf{X}|a^c_{\mathbf{i}}$ where $\mathbf{O_i}$ is on $a_{\mathbf{i}}\leq 1$.   
        By localization, $\mathbf{X}|a^c$ is conditionally compact with respect to the conditional topology $\mathcal{T}|a^c$, and $(\mathbf{O_i}|a^c)$ is a conditional open covering of $\mathbf{X}|a^c$.
        By assumption, there exists a conditionally finite subfamily $(\mathbf{O_{i_l}}|a^c)_{\mathbf{1}\leqslant \mathbf{l}\leqslant \mathbf{n}}$ such that 
            $\mathbf{X}|a^c=\sqcup_{\mathbf{1}\leqslant \mathbf{l}\leqslant \mathbf{n}} \mathbf{O_{i_l}}|a^c$. 
            It follows from de Morgan's law
            \begin{align*}
            \mathbf{X}|a&=(\mathbf{X}|a^c)^\sqsubset=\left(\sqcup_{\mathbf{1}\leqslant \mathbf{l}\leqslant \mathbf{n}} \mathbf{O_{i_l}}|a^c\right)^\sqsubset=\sqcap_{\mathbf{1}\leqslant \mathbf{l}\leqslant \mathbf{n}} (\mathbf{O_{i_l}}|a^c)^\sqsubset \\&= \sqcap_{\mathbf{1}\leqslant \mathbf{l}\leqslant \mathbf{n}} (\mathbf{F_{i_l}}|a^c + \mathbf{X}|a)=\left(\sqcap_{\mathbf{1}\leqslant \mathbf{l}\leqslant \mathbf{n}} \mathbf{F_{i_l}}\right)|a^c + \mathbf{X}|a. 
            \end{align*}
            By the conditional finite intersection property, $\sqcap_{\mathbf{1}\leqslant \mathbf{l}\leqslant \mathbf{n}} \mathbf{F_{i_l}}$ is on $1$, and therefore $(\sqcap_{\mathbf{1}\leqslant \mathbf{l}\leqslant \mathbf{n}} \mathbf{F_{i_l}})|a^c$ is on $a^c$. 
            Thus it holds $\mathbf{X}|a=(\sqcap_{\mathbf{1}\leqslant \mathbf{l}\leqslant \mathbf{n}} \mathbf{F_{i_l}})|a^c+\mathbf{X}|a$ if, and only if, $a^c=0$ which is the desired contradiction. 
        \item[] To show that \ref{cond:comp05} implies \ref{cond:comp01}, let $(\mathbf{O}_{\mathbf{i}})$ be a conditional open covering of $\mathbf{X}$. 
        Let 
        \begin{equation*}
        b:=\vee\{a: \sqcup_{\mathbf{1\leqslant l\leqslant n}}\mathbf{O_{i_l}}|a=\mathbf{X}|a \text{ for some conditionally finite subfamily } (\mathbf{O_{i_l}})\}.
        \end{equation*}
        By consistency, the well-ordering theorem and stability, $b$ is attained by some $(\mathbf{O_{i_l}})$. 
        By contradiction, suppose that $b<1$. 
        Up to localization, we may assume that $b^c=1$, otherwise the following argument is done on $b^c<1$. 
        Then for all conditionally finite subfamilies $(\mathbf{O_{i_l}})$, it holds 
        \begin{equation*}
        \left(\sqcup_{\mathbf{1\leqslant l\leqslant n}}\mathbf{O_{i_l}}\right)|a\neq \mathbf{X}|a, \quad \text{ for all } 0<a\leq 1.  
        \end{equation*}
        Hence $\left(\sqcup_{\mathbf{1\leqslant l\leqslant n}}\mathbf{O_{i_l}}\right)^\sqsubset=\sqcap_{\mathbf{1\leqslant l\leqslant n}}\mathbf{O_{i_l}}^\sqsubset=\sqcap_{\mathbf{1\leqslant l\leqslant n}}\mathbf{F_{i_l}}$ is on $1$, and thus $(\mathbf{F_i})$ satisfies the conditional finite intersection property. 
        By assumption, $\sqcap \mathbf{F_i}$ is on $1$. 
        However, this implies that $(\sqcap \mathbf{F_i})^\sqsubset=\sqcup \mathbf{O_i}\neq \mathbf{X}$, contrary to the assumption. 
        \item[] As for \ref{cond:comp03} implies \ref{cond:comp05}, let $(\mathbf{F_{\mathbf{i}})}$ be a conditional family of conditional closed subsets of $\mathbf{X}$ satisfying the conditional finite intersection property.
            Then $\mathcal{G}:=\{\mathbf{\sqcap_{1\leqslant l\leqslant n} F_{i_l}}\colon \mathbf{(F_{i_l})_{1\leqslant l\leqslant n}} \text{ conditionally finite} \}$ is a conditional filter base. 
            By assumption, there exists a conditional element $\mathbf{x}$ of $\Lim(\mathcal{F})$. 
            Thus $\sqcap \mathbf{F_{i}}$ is on $1$.  
        \item[] To show that \ref{cond:comp05} implies  \ref{cond:comp03}, let $\mathcal{F}$ be a conditional filter on $\mathbf{X}$. 
            Since $\Cl(\sqcap \mathbf{Y}_\mathbf{i})\sqsubseteq \sqcap \Cl(\mathbf{Y}_{\mathbf{i}})$, it follows that $\set{\Cl(\mathbf{Y}):\mathbf{Y} \in \mathcal{F}}$ is a stable collection of conditional closed subsets of $\mathbf{X}$ fulfilling the conditional finite intersection property.  
            Hence $\Lim \mathcal{F}$ is on $1$, and thus $\mathcal{F}$ has a conditional cluster point. 
    \end{enumerate}
\end{proof}
\begin{proposition}\label{p:compactsets}
    Let $(\mathbf{X},\mathcal{T})$ and $(\mathbf{X}^\prime,\mathcal{T}^\prime)$ be conditional topological spaces, $\mathbf{f}: \mathbf{X}\to \mathbf{X}^\prime$  a conditionally continuous function, and $\mathbf{Y}$  a conditional compact subset of $\mathbf{X}$ on $1$. 
    Then $\mathbf{f}(\mathbf{Y})$ is a conditional compact subset of $\mathbf{X}^\prime$. 
\end{proposition}
\begin{proof}
Let $(\mathbf{O_i})$ be a conditional open covering of $\mathbf{f(Y)}$, that is 
$\mathbf{f(Y)}\sqsubseteq \sqcup \mathbf{O_i}$. 
By \eqref{f:021}, \eqref{f:031} and \eqref{f:061}, one has 
\begin{equation*}
 \mathbf{Y}\sqsubseteq \mathbf{f}^{-1}(\mathbf{f}(\mathbf{Y}))\sqsubseteq \mathbf{f}^{-1}(\sqcup \mathbf{O_i})=\sqcup \mathbf{f}^{-1}(\mathbf{O_i}). 
 \end{equation*} 
 By Proposition \ref{p:cts}, it follows that $(\mathbf{f}^{-1}(\mathbf{O_i}))$ is a conditional open covering of $\mathbf{Y}$. 
By assumption, there exists a conditionally finite subfamily $(\mathbf{f(O_{i_l})})_{\mathbf{1\leqslant l\leqslant n}}$ such that $\mathbf{Y}\sqsubseteq \sqcup_{\mathbf{1\leqslant l\leqslant n}} \mathbf{f}(\mathbf{O_{i_l}})$. 
By \eqref{f:021}, \eqref{f:031} and \eqref{f:061}, one has 
\begin{equation*}
\mathbf{f(Y)}\sqsubseteq \mathbf{f}(\sqcup_{\mathbf{1\leqslant l\leqslant n}}\mathbf{f}^{-1}(\mathbf{O_{i_l}}))= \sqcup_{\mathbf{1\leqslant l\leqslant n}} \mathbf{f}(\mathbf{f}^{-1}(\mathbf{O_{i_l}}))\sqsubseteq \sqcup_{\mathbf{1\leqslant l\leqslant n}} \mathbf{O_{i_l}}. 
\end{equation*}
\end{proof}
\begin{proposition}\label{p:closedcompact}
Let $(\mathbf{X},\mathcal{T})$ be a conditionally compact space and $\mathbf{Y}$  a conditionally closed subset of $\mathbf{X}$ on $1$. 
Then $\mathbf{Y}$ is conditionally compact. 
\end{proposition}
\begin{proof}
Without loss of generality, we may assume that $\mathbf{Y}^\sqsubset$ is on $1$, since otherwise $\mathbf{Y}|a^c=\mathbf{X}|a^c$ is already conditionally compact by localization, where $\mathbf{Y}^\sqsubset$ is on $a<1$. 
Let $(\mathbf{O_i})$ be a conditional open cover of $\mathbf{Y}$. 
Then $(\mathbf{\hat{O}_i})$ where $\mathbf{\hat{O}_i}:=\mathbf{O_i}\sqcup \mathbf{Y}^\sqsubset$ is a conditional open cover of $\mathbf{X}$. 
By assumption, there exists a conditionally finite subfamily $(\mathbf{\hat{O}_{i_l}})_{\mathbf{1\leqslant l \leqslant n}}$ conditionally covering $\mathbf{X}$. 
Thus $(\mathbf{O_{i_l}})_{\mathbf{1\leqslant l \leqslant n}}$ is a conditional open cover of $\mathbf{Y}$. 
\end{proof}
We finish this section with a conditional Tychonoff's theorem. 
\begin{theorem}\label{thm:tychonoff}
    Let $(\mathbf{X}_i,\mathcal{T}_i)$ be a non-empty family of conditional topological spaces and let $\mathbf{X}=\prod \mathbf{X}_i$ be  endowed with the conditional product topology.
    Then $\mathbf{X}$ is conditionally compact if, and only if, $\mathbf{X}_i$ is conditionally compact for every $i$. 
\end{theorem}
\begin{proof}
    Every conditional projection $\mathbf{pr}_i:\mathbf{X\to X}_i$ is conditionally continuous.
    Therefore, if $\mathbf{X}$ is conditionally compact, so is $\mathbf{X}_i=\mathbf{pr}_i(\mathbf{X})$ for every $i$ due to Proposition \ref{p:compactsets}.
    Conversely, assume that $\mathbf{X}_i$ is conditionally compact for each $i$ and let $\mathcal{U}$ be a conditional ultrafilter on $\mathbf{X}$.
    It follows from Proposition \ref{prop:ultrafilterfonct} that $\mathbf{pr}_i(\mathcal{U})$ is a conditional ultrafilter base on $\mathbf{X}_i$ for each $i$.  
    Since $\mathbf{X}_i$ is conditionally compact, $\mathbf{pr}_i(\mathcal{U})\to \mathbf{x}_i$ for some conditional element $\mathbf{x}_i$ of $\mathbf{X}_i$ for each $i$ due to Proposition \ref{characcomp}. 
    Let $\mathbf{O}_i$ be a conditional open neighborhood of $\mathbf{x}_i$ for some $i$. 
    Then there exists $\mathbf{U}\in\mathcal{U}$ such that $\mathbf{pr}_i(\mathbf{U})\sqsubseteq \mathbf{O}_i$. 
    By \eqref{f:021} and \eqref{f:061}, one has $\mathbf{U}\sqsubseteq \mathbf{pr}_i^{-1}(\mathbf{O}_i)$, and thus $\mathbf{pr}_i^{-1}(\mathbf{O}_i)\in \mathcal{U}$. 
    By stability of $\mathcal{U}$, it follows that any conditional neighborhood of $\mathbf{x}:=(\mathbf{x}_i)$ is an element in $\mathcal{U}$ which shows that $\mathcal{U}\to \mathbf{x}$. Proposition \ref{characcomp} implies that $\mathbf{X}$ is conditionally compact. 
\end{proof}
\section{Conditional real numbers and metric spaces}\label{ch:real}
\begin{definition}
    We call a conditional set $\mathbf{X}$ together with a conditional function $\mathbf{+}:\mathbf{X}\times \mathbf{X}\to \mathbf{X}$ a \emph{conditional group} if $(X,+)$ is a classical group.  
    A conditional set $\mathbf{X}$ with two conditional functions $\mathbf{+}:\mathbf{X}\times \mathbf{X}\to \mathbf{X}$ and $\mathbf{\cdot}:\mathbf{X}\times \mathbf{X}\to \mathbf{X}$ is a \emph{conditional ring} if $(X,+,\cdot)$ is a classical ring. 
    We denote by $\mathbf{0}$ and $\mathbf{1}$ the conditional neutral elements for $+$ and $\cdot$, respectively.\footnote{Whenever there is no risk of confusion, we use $+$ for addition and concatenations, and $0,1$ denote the neutral elements and the distinguished elements $0,1\in\mathcal{A}$.} 
    Let $(\mathbf{X},\mathbf{+},\mathbf{\cdot})$ be a conditional ring and $(\mathbf{X},\leqslant)$  a conditional totally ordered set. 
    We say that $\mathbf{X}$ is a \emph{conditional ordered ring} if $\mathbf{x}<\mathbf{y}$ implies $\mathbf{x}+\mathbf{z}<\mathbf{y} +\mathbf{z}$ and $\mathbf{x},\mathbf{y}>\mathbf{0}$ implies $\mathbf{x}\cdot \mathbf{y}>\mathbf{0}$ for all conditional elements $\mathbf{x,y,z}$ of $\mathbf{X}$. 
    A conditional ring $(\mathbf{X},\mathbf{+},\mathbf{\cdot})$ is a \emph{conditional field} if for every conditional element $\mathbf{x}$ of $\mathbf{X}^\ast:= \set{\mathbf{0}}^\sqsubset$ there exists a conditional element $\mathbf{y}$ of $\mathbf{X}^\ast$ such that $\mathbf{x}\cdot \mathbf{y}=\mathbf{y}\cdot \mathbf{x}=\mathbf{1}$. 
\end{definition}
For a conditional ordered field $(\mathbf{K},+,\cdot,\leqslant)$, define $\mathbf{K}_+:=\set{\mathbf{x}: \mathbf{x\geqslant 0}}$ and $\mathbf{K}_{++}:=\set{\mathbf{x}: \mathbf{x> 0}}$. 
The conditional absolute value $\abs{\cdot}:\mathbf{K}\to \mathbf{K}_+$ is generated by the stable function $|x|:=\max\set{x,-x}$.
\begin{examples}
    \begin{enumerate}[fullwidth, label=\arabic*)]
        \item Let $(S,+, \cdot)$ be a classical ring. The conditional functions generated by $+$ and $\cdot$ define a conditional ring structure on the conditional set $\mathbf{S}$ generated by $S$ in the sense of Example \ref{ex:L0} 5).   
            For instance, the conditional ring of conditional rational numbers $(\mathbf{Q},+,\cdot)$ is generated by the ring of rational numbers $(\mathbb{Q},+,\cdot)$.   
            Inspection shows that $\mathbf{Q}$ is a conditional ordered field where the conditional order is the one defined in Examples \ref{ex:order} 1). 
        \item Let $\mathbf{Q}^\mathbf{N}$ be the conditional set of conditional sequences of conditional elements of $\mathbf{Q}$.
            On $\mathbf{Q}^\mathbf{N}$ define 
            \begin{align*}
                \mathbf{(q_n) + (p_n)}&:= \mathbf{(q_n + p_n), \quad (q_n)\cdot (p_n):=(q_n\cdot p_n), }\\
                (\mathbf{q_n})& \leqslant (\mathbf{p_n)} \text{ whenever } \mathbf{q_n\leqslant p_n} \text{ for each } \mathbf{n}. 
            \end{align*}
            Inspection shows that $(\mathbf{Q}^\mathbf{N},+,\cdot,\leqslant)$ is a conditional ordered ring. 
    \end{enumerate}
\end{examples}
Endow $\mathbf{Q}$ with the conditional Euclidean topology, see Example \ref{exep:Qtop}.  
A conditional sequence $(\mathbf{q_n})$ of conditional elements of $\mathbf{Q}$ is said to be \emph{conditionally Cauchy} if for every conditional element $\mathbf{r}$ of $\mathbf{Q}_{++}$ there exists $\mathbf{n}_0$ such that $\abs{\mathbf{q_n}-\mathbf{q_m}}<\mathbf{r}$ for all $\mathbf{m},\mathbf{n}\geqslant \mathbf{n}_0$.
Denote by $\mathbf{C}$ the conditional subset of $\mathbf{Q}^\mathbf{N}$ consisting of all conditional Cauchy sequences. 
Let $(\mathbf{p_n})\sim (\mathbf{q_n})$ whenever $(\mathbf{q_n}-\mathbf{p_n})\to 0$ be a conditional equivalence relation on $\mathbf{C}$. 
On $\mathbf{R}:=\mathbf{C}/\sim$ define 
\begin{align*}
    [(\mathbf{q_n})]+[(\mathbf{p_n})]&:=[(\mathbf{q_n})+(\mathbf{p_n})],\quad [(\mathbf{q_n})]\cdot[(\mathbf{p_n})]:=[(\mathbf{q_n})\cdot(\mathbf{p_n})],\\
    [(\mathbf{p_n})]\leqslant [(\mathbf{q_n})] \text{ whenever for all }  &\mathbf{r} \text{ of } \mathbf{Q}_{++} \text{ there exists } \mathbf{n}_0 \text{ such that } \mathbf{q_n}-\mathbf{p_n}>-\mathbf{r} \text{ for all } \mathbf{n}\geqslant \mathbf{n}_0.  
\end{align*}
Inspection shows that $(\mathbf{R},+,\cdot, \leqslant)$ is a conditional ordered field. 
\begin{definition}
    We call $\mathbf{R}=(\mathbf{R},+,\cdot, \leqslant)$ the conditional ordered field of \emph{conditional real numbers}.
\end{definition}
It can be checked that  
\begin{enumerate}[label=(\roman*)]
    \item $\mathbf{R}$ is conditionally Dedekind complete; 
    \item for every conditional element $\mathbf{x}$ of $\mathbf{R}$ there exists a conditional element $\mathbf{n}$ of $\mathbf{N}$ such that $\mathbf{n}>\mathbf{x}$;  
    \item for conditional elements $\mathbf{x},\mathbf{r}$ of $\mathbf{R}$ such that $\mathbf{r>0}$, let $\mathbf{B}_{\mathbf{r}}(\mathbf{x})=\set{\mathbf{y}: \abs{\mathbf{x}-\mathbf{y}}<\mathbf{r}}$. 
        The conditional collection $\mathcal{B}$ of conditional sets $\set{\mathbf{B}_{\mathbf{r}}(\mathbf{x})\colon \mathbf{x}, \mathbf{r} \text{ of }\mathbf{R}\text{ with }\mathbf{r>0}}$ is a conditional topological base of a conditional Hausdorff topology.
        $\mathbf{R}$ endowed with $\mathcal{T}^\mathcal{B}$ is conditionally separable and complete. 
        We call $\mathcal{T}^\mathcal{B}$ the \emph{conditional Euclidean topology} on $\mathbf{R}$.  
\end{enumerate}
The following theorem connects the conditional analysis in $L^0$-modules \citep{Cheridito2012,kupper03,DKKM13} to conditional set theory. 
For a proof of the isomorphism of $L^0$ and the real numbers in the universe of Boolean-valued sets over the measure algebra associated to a $\sigma$-finite measure space see \cite[Theorem 7.1]{bell2005set}. 
The conditional set $\mathbf{L^0}$ is constructed in Example \ref{ex:L0} 4), and its conditional topology is defined in Example \ref{e:l0top}. 
\begin{theorem}
    Let $\mathcal{A}$ be the measure algebra associated to a $\sigma$-finite measure space $(\Omega,\mathcal{F},\mu)$. 
    Then there exists a conditional bijection from $\mathbf{R}$ to $\mathbf{L^0}$. 
\end{theorem}
\begin{proof} 
For $\sum q_n|a_n \in Q$, one has $\sum_{n\geq 1} q_n1_{A_n}:=\lim_{n\to \infty}(q_1 1_{A_1}+\ldots+q_n 1_{A_n})\in L^0$ where $a_n=[A_n]$ for each $n$. 
Let $\mathbf{x}$ be a conditional element of $\mathbf{L^0}$. 
Since $\mathbf{L^0}$ is conditionally separable, there exists a conditional sequence $(\mathbf{q_n})$ of conditional elements of $\mathbf{L^0}$ with $q_n=\sum_{m \geq 1} q_{n,m} 1_{A_{n,m}}$ where $(q_{n,m})$ is a sequence in $\mathbb{Q}$ such that $\mathbf{q_n}\to \mathbf{x}$, see \cite[Lemma 5.3.2]{jamneshan13}. 
Conversely, if $(\mathbf{q_n})$ is a conditional Cauchy sequence of conditionally rational-valued functions, then $(q_n)$ is a Cauchy net in $L^0$ endowed with the $L^0$-topology. 
By a conditional Bolzano-Weierstra\ss \, theorem, see \cite[Lemma 1.64]{foellmer01} or \cite[Theorem 3.8]{Cheridito2012}, one has that $(q_n)$ converges in the $L^0$-topology.\footnote{In the cited results the convergence is almost everywhere which implies convergence in $L^0$-topology for stable sequences.} 
From Proposition \ref{p:konvergenzobenunten} it follows that $(\mathbf{q_n})$ conditionally converges to $\mathbf{x}$. 
Then $\mathbf{f}:\mathbf{R}\to \mathbf{L^0}$ defined by the stable function $f(q_n)=\lim (\sum_{m\geq 1} q_{n,m} 1_{A_{n,m}})$ is a conditional function. 
It is conditionally injective since $\mathbf{L^0}$ is conditionally Hausdorff, and conditionally surjective due to the conditional separability of $\mathbf{L^0}$. 
\end{proof}
\begin{definition}
    Let $\mathbf{X}$ be a conditional set.
    A \emph{conditional metric} is a conditional function $\mathbf{d}:\mathbf{X}\times \mathbf{X}\to \mathbf{R}_+$ such that
    \begin{enumerate}[label=\textit{(\roman*)}]
        \item\label{cond:metric01} $\mathbf{d}(\mathbf{x},\mathbf{y})=\mathbf{0}$ if, and only if, $\mathbf{x}=\mathbf{y}$,
        \item\label{cond:metric02} $\mathbf{d}(\mathbf{x},\mathbf{y})=\mathbf{d}(\mathbf{y},\mathbf{x})$ for all conditional elements $\mathbf{x}$, $\mathbf{y}$ of $\mathbf{X}$,
        \item\label{cond:metric03} $\mathbf{d}(\mathbf{x},\mathbf{z})\leqslant \mathbf{d}(\mathbf{x},\mathbf{y})+\mathbf{d}(\mathbf{y},\mathbf{z})$ for all conditional elements $\mathbf{x},\mathbf{y},\mathbf{z}$ of $\mathbf{X}$.
    \end{enumerate}
    The pair $(\mathbf{X},\mathbf{d})$ is called a \emph{conditional metric space}. 
\end{definition}
Given a conditional metric space $(\mathbf{X},\mathbf{d})$, we define a conditional topological base $\mathcal{B}$ consisting of conditional balls $\mathbf{B}_\mathbf{r}(\mathbf{x})$ and conditional Cauchy sequences similarly to the respective definitions for $\mathbf{Q}$, see Example \ref{exep:Qtop}. 
Inspection shows that the conditional topology $\mathcal{T}^\mathcal{B}$ is conditionally Hausdorff and first countable.  
We say that $(\mathbf{X},\mathbf{d})$ is \emph{conditionally complete}, if every conditional Cauchy sequence has a conditional limit, and \emph{conditionally sequentially compact}, if every conditional sequence has a conditional cluster point. 
A conditional subset $\mathbf{Y}$ of $\mathbf{X}$ on $1$ is \emph{conditionally totally bounded}, if for every conditional element $\mathbf{r}$ of $\mathbf{R}_{++}$ there exists a conditionally finite family $(\mathbf{x_l)_{1\leqslant l\leqslant n}}$ of conditional elements of $\mathbf{Y}$, called a \emph{conditional $\mathbf{r}$-net}, such that $\sqcup_{\mathbf{1\leqslant l\leqslant n}} \mathbf{B_{r}(x_l)}=\mathbf{Y}$.  
A conditionally totally bounded metric space is conditionally separable.
Indeed, for every conditional element $\mathbf{r}$ of $\mathbf{Q}_{++}$ choose $(\mathbf{x^r_l)_{1\leqslant l\leqslant m_{\mathbf{r}}}}$ such that $\mathbf{\sqcup_{1\leqslant l \leqslant m_{\mathbf{r}}} \mathbf{B_{r}(x^r_l)}=\mathbf{X}}$. 
Then $\mathbf{\sqcup_{r \text{ of } \mathbf{Q}_{++}}\set{\mathbf{x^r_l: 0\leqslant l\leqslant m_r}}}$ is a conditionally countable dense subset by Proposition \ref{lem:repfin}. 

We characterize conditional compactness in conditional metric spaces by a conditional Borel-Lebesgue theorem.   
\begin{theorem}\label{thm:metriccompat}
    For a conditional metric space $(\mathbf{X},\mathbf{d})$, the following are equivalent:
    \begin{enumerate}[label=\textit{(\roman*)}]
        \item\label{met:comp} $\mathbf{X}$ is conditionally compact.
        \item\label{met:decreas} If $(\mathbf{F_n})$ is a conditionally decreasing\footnote{That is, $\mathbf{m}\leqslant \mathbf{n}$ implies $\mathbf{F_n}\sqsubseteq \mathbf{F_m}$.} family of conditionally closed sets in $\mathbf{X}$, then $\sqcap \mathbf{F_n}$ is on $1$. 
        \item\label{met:seccomp} $\mathbf{X}$ is conditionally sequentially compact.
        \item\label{met:bound} $\mathbf{X}$ is conditionally complete and totally bounded. 
    \end{enumerate}
\end{theorem}
\begin{proof}
    \begin{enumerate}[fullwidth]
        \item[] As for \ref{met:comp} implies \ref{met:decreas}, note that $\sqcap_{\mathbf{1}\leqslant \mathbf{l}\leqslant \mathbf{m}}\mathbf{F_{n_l}}=\mathbf{F_{n_m}}$ for any conditionally finite subfamily $(\mathbf{F_{n_l}})_{\mathbf{1}\leqslant\mathbf{l}\leqslant \mathbf{m}}$. Thus $\sqcap_{\mathbf{1}\leqslant \mathbf{l}\leqslant \mathbf{m}}\mathbf{F_{n_l}}$ is on $1$. Proposition \ref{characcomp} implies that $\sqcap \mathbf{F_n}$ is on $1$. 
        \item[] To show that \ref{met:decreas} implies \ref{met:seccomp}, let $(\mathbf{x_n})$ be a conditional sequence of conditional elements of $\mathbf{X}$. For each $n\in N$, set $E_n:=\{x_m: m\geq n\}$. Then $(E_n)$ is a stable family of stable subsets of $X$, and thus $(\mathbf{E_n})$ is a conditional family of conditional subsets of $\mathbf{X}$. 
        Let $\mathbf{F_n}:=\Cl(\mathbf{E_n})$ for each conditional element $\mathbf{n}$ of $\mathbf{N}$. 
        Then $(\mathbf{F_n})$ is a conditionally decreasing family of conditional closed sets. 
        By assumption, there exists a conditional element $\mathbf{x}$ of $\sqcap \mathbf{F_n}$. 
        Now set $\mathbf{x_{n_1}}=\mathbf{x_1}$. 
        Supposing that we have already chosen $\mathbf{x_{n_1}},\ldots,\mathbf{x_{n_{k-1}}}$, choose a conditional element $\mathbf{x_{n_k}}$ of $\mathbf{F_{n_{k-1}+1}}$ such that $\mathbf{d(x_{n_k},x)}<\mathbf{1/k}$. 
        For each $k=\sum n_i|a_i\in N$, set $\mathbf{x_{n_k}}:=\sum \mathbf{x_{n_i}}|a_i$. 
    Inspection shows that $(\mathbf{x_{n_k}})$ is a conditional subsequence of $(\mathbf{x_n})$ conditionally converging to $\mathbf{x}$. 
    \item[] We show that \ref{met:seccomp} implies \ref{met:bound}. 
    Conditional completeness follows by the conditional triangle inequality.  
    As for the conditional total boundedness, suppose for the sake of contradiction that $b:=\vee M>0$ where 
    \begin{equation*}
   M=\{a: \sqcup_{\mathbf{1}\leqslant \mathbf{l}\leqslant \mathbf{n}}\mathbf{B_r(x_l)}|a^\prime \neq \mathbf{X}|a^\prime \text{ for all }  0 < a^\prime\leq a \text{ for some }\mathbf{r>0}\text{ and all }  (\mathbf{x_l})_{\mathbf{1\leqslant l\leqslant n}} \text{ of } \mathbf{X} \}. 
    \end{equation*}
    Note that if $a\in M$, then $a^\prime\in M$ whenever $a^\prime\leq a$. 
    Index $M$ by $(a_i)$. 
    By the well-ordering theorem, there exists $(b_i)\in p(b)$ with $b_i\leq a_i$ for all $i$. 
    Then $b_i\in M$ for each $i$.
    Let $\mathbf{r}_i>\mathbf{0}$ satisfy the condition in $M$ for $b_i$ for each $i$. 
    By stability, $b$ is attained in $M$ for $\mathbf{r}=\sum \mathbf{r}_i|b_i$.  
    Without loss of generality, assume that $b=1$. 
    Let $\mathbf{x_1}$ be any conditional element of $\mathbf{X}$. 
    Suppose we have found a conditionally finite family $(\mathbf{x_l})_{\mathbf{1\leqslant l\leqslant n}}$ such that $\mathbf{d(x_{l_1},x_{l_2})}\geqslant \mathbf{r}$ for all $\mathbf{1\leqslant l_1,l_2\leqslant n}$ with $\mathbf{l_1\sqcap l_2}=\mathbf{N}|0$.  
     Let 
    \begin{equation*}
    c:=\vee\{a: \mathbf{d}(\mathbf{x_l},\mathbf{x})|a\geqslant \mathbf{r}|a \text{ for some }\mathbf{x} \text{ of } \mathbf{X} \text{ and for all }\mathbf{1\leqslant l\leqslant n}\}. 
    \end{equation*}
    By consistency, the well-ordering theorem and stability, $c$ is attained by some $\mathbf{x_{n+1}}:=\mathbf{x}$. 
     It holds $c=1$, since otherwise there exists $a>0$ such that $\sqcup_{\mathbf{1\leqslant l \leqslant n}} \mathbf{B_r(x_l)}|a=\mathbf{X}|a$ which contradicts the maximality of $b$. 
     For $\mathbf{1\leqslant l\leqslant n+1}$, put $\mathbf{x_l}:=\sum \mathbf{x_{n_i}}|a_i$ where $l=\sum n_i|a_i$.  
    By induction, we obtain a conditional sequence $(\mathbf{x_n})$ such that $\mathbf{d(x_{n_1},x_{n_2})}\geqslant \mathbf{r}$ whenever $\mathbf{n_1\sqcap n_2}=\mathbf{N}|0$. 
    By construction, $(\mathbf{x_n})$ does not have a conditionally converging subsequence. 
    Indeed, if there exists a conditional subsequence $(\mathbf{x_{n_k}})$ conditionally converging to some $\mathbf{x}$ of $\mathbf{X}$, then there exists $\mathbf{k}_0$ of $\mathbf{N}$ such that $\mathbf{d(x_{n_k},x)}<\mathbf{r/2}$ for all $\mathbf{k}\geqslant \mathbf{k}_0$. 
    By the conditional triangle inequality, one has $\mathbf{d}(\mathbf{x}_{\mathbf{n}_{\mathbf{k}_0+\mathbf{1}}},\mathbf{x}_{\mathbf{n}_{\mathbf{k}_0+\mathbf{2}}})< \mathbf{r}$ contrary to the definition of $(\mathbf{x_n})$. 
    \item[] To show that \ref{met:bound} implies \ref{met:comp}, suppose, by contradiction, that $b:=\vee M>0$ where  
    \begin{equation*}
    M=\{a: \sqcup_{\mathbf{1}\leqslant \mathbf{l}\leqslant \mathbf{n}} \mathbf{U_{i_l}}|a^\prime\neq \mathbf{X}|a^\prime \text{ for all } 0<a^\prime\leq a \text{ and } (\mathbf{U_{i_l}}) \text{ of some conditional open covering } (\mathbf{U_i})\}. 
    \end{equation*}
   By consistency, the well-ordering theorem and stability, $b$ is attained by some $(\mathbf{U_i})$. 
   By localization, we may assume that $b=1$. 
   Now for $\mathbf{r=1/2}$ there exists a conditional $\mathbf{1/2}$-net $(\mathbf{x}^1_{\mathbf{l}})_{\mathbf{1\leqslant l\leqslant n}_1}$ by assumption. 
   Let 
   \begin{equation*}
   c_1:=\vee\{a:  \sqcup_{\mathbf{1\leqslant k\leqslant m}}\mathbf{U_{i_k}}|a^\prime\neq \mathbf{B_{1/2}}(\mathbf{x}^1_\mathbf{l})|a^\prime \text{ for all } 0<a^\prime\leq a \text{ and } (\mathbf{U_{i_k}})_{\mathbf{1\leqslant k\leqslant m}} \text{ of } (\mathbf{U_i}) \text{ for some } \mathbf{x}^1_\mathbf{l}\}. 
   \end{equation*}
   By consistency, stability and the well-ordering theorem, $c_1$ is attained by some $\mathbf{y_1}:=\mathbf{x}^1_\mathbf{l}$. By assumption, it must hold $c_1=1$. Indeed, if $c_1^c>0$, then for all $\mathbf{x}^1_\mathbf{l}$ there exists $(\mathbf{U_{i_k}})_{\mathbf{1\leqslant k\leqslant m}}$ such that $\sqcup_{\mathbf{1\leqslant k\leqslant m}}\mathbf{U_{i_k}}|c^c_1=\mathbf{B_{1/2}}(\mathbf{x}^1_\mathbf{l})|c^c_1$. 
   Proposition \ref{lem:repfin} implies that $\mathbf{X}|c^c_1$ is conditionally covered by a conditionally finite subfamily of $(\mathbf{U_i}|c^c_1)$ which contradicts the maximality of $b$. 

  Next for $\mathbf{r=1/4}$, let $(\mathbf{x}^2_{\mathbf{l}})_{\mathbf{1\leqslant l\leqslant n}_2}$ be a conditional $\mathbf{1/4}$-net, and set $c=\vee M$ where $M$ is the collection of all $a\in\mathcal{A}$ such that there exists $\mathbf{x}^2_\mathbf{l}$ with $\mathbf{B_{1/2}(y_1)}\sqcap \mathbf{B_{1/4}}(\mathbf{x}^2_{\mathbf{l}})$ is on $1$ and 
 \begin{equation*}
   \sqcup_{\mathbf{1\leqslant k\leqslant m}} \mathbf{U_{i_k}}|a^\prime\neq \mathbf{B_{1/4}}(\mathbf{x}^2_\mathbf{l})|a^\prime \text{ for all } 0<a^\prime\leq a \text{ and } (\mathbf{U_{i_k}})_{\mathbf{1\leqslant k\leqslant m}} \text{ of } (\mathbf{U_i}). 
   \end{equation*} 
   By consistency, the well-ordering theorem and stability, $c_2$ is attained by some $\mathbf{y_2}:=\mathbf{x}^2_\mathbf{l}$. 
   Moreover, one has $c_2=1$. 
   Indeed, since $(\mathbf{B_{1/4}}(\mathbf{x}^2_\mathbf{l}))_{\mathbf{1\leqslant l\leqslant n}_2}$ is a conditional family, for each conditional element $\mathbf{x}$ of $\mathbf{B_{1/2}(y_1)}$ there exists $\mathbf{x}^2_\mathbf{l}$ such that $\mathbf{x}$ is a conditional element of $\mathbf{B_{1/4}}(\mathbf{x}^2_\mathbf{l})$ due to Lemma \ref{lem:1fcunion}.  
   Thus if $c_2^c>0$, then for all $\mathbf{x}^2_\mathbf{l}$ such that $\mathbf{B_{1/2}(y_1)}\sqcap \mathbf{B_{1/4}}(\mathbf{x}^2_\mathbf{l})$ is on $1$ there exists a conditionally finite subfamily $(\mathbf{U_{i_k}})$ of $(\mathbf{U_i})$ such that $\sqcup \mathbf{U_{i_k}}|c^c=\mathbf{B_{1/4}}(\mathbf{x}^2_\mathbf{l})|c_2^c$. 
   By Proposition \ref{lem:repfin}, this contradicts the maximality of $b$.  
   
   We continue analogously: at the $n$-th stage, with $\mathbf{r=1/2^n}$ let $\mathbf{y_n}$ be any conditional element of that conditional $\mathbf{1/2^n}$-net such that $\mathbf{B_{1/2^{n-1}}(y_{n-1}})\sqcap \mathbf{B_{1/2^{n}}(y_{n}})$ is on $1$, and having the property that if a conditional finite subfamily of $(\mathbf{U_i}|a)$ conditionally covers $\mathbf{B_{1/2^{n}}(y_{n}})|a$, then $a=0$. 

   For $n=\sum n_i|a_i$, set $\mathbf{y_n}=\sum\mathbf{y_{n_i}}|a_i$. 
   By construction, $(\mathbf{y_n})$ is a conditional Cauchy sequence of conditional elements of $\mathbf{X}$. 
   By assumption, $(\mathbf{y_n})$ conditionally converges to some conditional element $\mathbf{y}$ of $\mathbf{X}$. 
   Since $(\mathbf{U_i})$ conditionally covers $\mathbf{X}$, there exists some $\mathbf{i}_0$ such that $\mathbf{y}$ is a conditional element of $\mathbf{U}_{\mathbf{i}_0}$ by Lemma \ref{lem:1fcunion}.   
   Since $\mathbf{U}_{\mathbf{i}_0}$ is conditionally open there exists $\mathbf{r>0}$ such that $\mathbf{B_r(y)}\sqsubseteq \mathbf{U}_{\mathbf{i}_0}$. By the definition of $\mathbf{y}$, there exists $\mathbf{n}$ of $\mathbf{N}$ such that $\mathbf{d(y_n,y)}<\mathbf{r/2}$ and $\mathbf{1/2^n}<\mathbf{r/2}$. 
   It follows from the conditional triangle inequality that $\mathbf{B_{1/2^n}(y_n)}\sqsubseteq \mathbf{B_r(y)}\sqsubseteq \mathbf{U}_{\mathbf{i}_0}$. 
   But this contradicts the fact that there does not exist a conditionally finite subfamily of $(\mathbf{U_i})$ which is a conditional covering of $\mathbf{B_{1/2^n}(y_n)}$. 
    \end{enumerate}
\end{proof}
\section{Conditional topological vector spaces}\label{ch:vector}
\begin{definition} 
    A conditional group $\mathbf{X}$ is a \emph{conditional vector space} if there exists a conditional function $\cdot:\mathbf{R}\times \mathbf{X}\to \mathbf{X}$ such that $X$ is an $R$-module in the classical sense. 
    A conditional subset $\mathbf{Y}$ of $\mathbf{X}$ on $1$ is a \emph{conditional subspace} of $\mathbf{X}$ if  $\mathbf{Y}+\mathbf{Y}\sqsubseteq \mathbf{Y}$ and  $\mathbf{r} \mathbf{Y}\sqsubseteq \mathbf{Y}$ for every conditional element $\mathbf{r}$ of $\mathbf{R}$. 
    A conditional function $\mathbf{f}:\mathbf{X} \to \mathbf{R}$ is \emph{conditionally linear} if $\mathbf{    f(r x + y) = r f(x) + f(y)}$ for all conditional elements $\mathbf{x,y}$ of $\mathbf{X}$ and every conditional element $\mathbf{r}$ of $\mathbf{R}$. 
\end{definition}
\begin{definition}
Let $\mathbf{X}$ be a conditional vector space.
    Let $\mathbf{Y}$ be a conditional subset of $\mathbf{X}$ on $1$.
    Define
    \begin{equation*}
        \co(\mathbf{Y}):=\Set{\sum_{\mathbf{1\leqslant l\leqslant n}} \mathbf{r_l y_l} \colon \mathbf{(y_l)_{1\leqslant l\leqslant n}}  \text{ of } \mathbf{Y},\,\mathbf{0\leqslant r_l\leqslant 1},\,\mathbf{\sum_{1\leqslant l\leqslant n}r_l=1}, \,\mathbf{n} \text{ of } \mathbf{N}}.\footnote{We understand $\mathbf{\sum_{1\leqslant l\leqslant n} x_l}:=\sum(\sum_{l=1}^{n_i} \mathbf{x}_l)|a_i$ where $n=\sum n_i|a_i$.}
    \end{equation*}
    We say that $\mathbf{Y}$ is
    \begin{itemize}
    \item  \emph{conditionally} \emph{convex} if $\mathbf{Y}=\co(\mathbf{Y})$, 
    \item  \emph{conditionally} \emph{absorbing} if for any conditional element $\mathbf{x}$ of $\mathbf{X}$ there exists $\mathbf{r_x>0}$ such that $\mathbf{r x}$ is a conditional element of $\mathbf{Y}$ for every $\mathbf{r}$ of $\mathbf{R}$ with $\mathbf{\abs{r}\leqslant r_x}$, 
    \item  \emph{conditionally} \emph{circled} if $\mathbf{r y}$ is a conditional element of $\mathbf{Y}$ for every conditional element $\mathbf{y}$ of $\mathbf{Y}$ and all conditional elements $\mathbf{r}$ of $\mathbf{R}$ with $\mathbf{\abs{r}\leqslant 1}$. 
    \end{itemize}    
    A conditional function $\mathbf{f:X\to R}$ is \emph{conditionally convex} if $$\mathbf{    f(r x +(1-r) y) \leqslant r f(x)+(1-r)f(y)}$$ for all conditional elements $\mathbf{x,y}$ of $\mathbf{X}$ and $\mathbf{0\leqslant r\leqslant 1}$. 
\end{definition}
We prove a conditional Hahn-Banach theorem. 
\begin{theorem}
Let $\mathbf{X}$ be a conditional vector space, $\mathbf{Y}$  a conditional subspace of $\mathbf{X}$ and $\mathbf{f}:\mathbf{Y}\to \mathbf{R}$  a conditional linear function. 
    If $\mathbf{f(x)}\leqslant \mathbf{g(x)}$ for all conditional elements $\mathbf{x}$ of $\mathbf{Y}$ for some conditional convex function $\mathbf{g}:\mathbf{X}\to \mathbf{R}$, then there exists a conditional linear function $\mathbf{\hat{f}}:\mathbf{X} \to \mathbf{R}$ such that $\mathbf{\hat{f}(x)}\leqslant \mathbf{g(x)}$ for all conditional elements $\mathbf{x}$ of $\mathbf{X}$ and $\mathbf{\hat{f}(x)}=\mathbf{f(x)}$ for all conditional elements $\mathbf{x}$ of $\mathbf{Y}$. 
    \label{thm:hahnbanach}
\end{theorem}
\begin{proof}
    Let $\mathcal{E}$ be the collection of all pairs $\mathbf{(h,H)}$ where $\mathbf{H}$ is a conditional subspace of $\mathbf{X}$ with $\mathbf{Y}\sqsubseteq \mathbf{H}$ and $\mathbf{h}:\mathbf{H} \to \mathbf{R}$ is a conditionally linear function such that $\mathbf{h}=\mathbf{f}$ on $\mathbf{Y}$ and $\mathbf{h}\leqslant \mathbf{g}$ on $\mathbf{H}$.  
    Define a partial order on $\mathcal{E}$ by $\mathbf{(h,H)}\leq \mathbf{(h^\prime, H^\prime)}$ whenever $\mathbf{H\sqsubseteq H^\prime}$ and $\mathbf{h^\prime=h}$ on $\mathbf{H}$. 
    Let $(\mathbf{h}_i,\mathbf{H}_i)$ be a chain in $\mathcal{E}$. 
    By Proposition \ref{rule02}, it holds $(\mathbf{h},\mathbf{H})\in\mathcal{E}$ where $\mathbf{H}:=\sqcup \mathbf{H}_i$ and $\mathbf{h:\mathbf{H}\to R}$ is defined by $\mathbf{h(x)}=\sum \mathbf{h}_{i_j}(\mathbf{x}_{j})|a_{i_j}$ for every $\mathbf{x}=\sum \mathbf{x}_j|a_j$ of $\mathbf{H}$ where $(a_j)\in p(1)$ and $\mathbf{x}_{j}$ of $\mathbf{H}_{i_j}$ for each $j$. 
    By Zorn's lemma, there exists a maximal element $(\mathbf{\hat{f},\hat{H}})$ in $\mathcal{E}$.  

    By contradiction, suppose that $\mathbf{\hat{H}}^{\sqsubset}$ is on $a>0$. 
    By localization, we may assume that $a=1$. 
    Pick some $\mathbf{v}$ of $\mathbf{\hat{H}}^{\sqsubset}$ and put $\mathbf{\tilde{H}=\{x+rv: x \text{ of } \hat{H}, r\text{ of } R\}}$.  
    Inspection shows that $\mathbf{\tilde{H}}$ is a conditional subspace of $\mathbf{X}$ with $\mathbf{Y\sqsubseteq \hat{H}\sqsubset \tilde{H}\sqsubseteq X}$. 
    Every conditional element $\mathbf{y}$ of $\mathbf{\tilde{H}}$ is of the form $\mathbf{y=x+rv}$ for unique $\mathbf{x}$ of $\mathbf{\hat{H}}$ and $\mathbf{r}$ of $\mathbf{R}$.  
    Indeed, let $\mathbf{y=x+r v=\bar{x}+\bar{r}v}$ and suppose that $b=\vee \set{a:\mathbf{r}|a= \bar{\mathbf{r}}|a}<1$. 
    Since $\mathbf{(\bar{r}-r) v=x-\bar{x}}$ is a conditional element of $\mathbf{\hat{H}}$, it holds that $\mathbf{v}\sqcap \mathbf{\hat{H}}$ is on $b^c>0$ which contradicts the choice of $\mathbf{v}$. 
    Hence $b=1$, and thus $\mathbf{r}=\bar{\mathbf{r}}$ and $\mathbf{x=\bar{x}}$. 

    Any linear extension $\mathbf{\tilde{f}}$ of $\mathbf{\hat{f}}$ to $\mathbf{\tilde{H}}$ has to fulfill $\mathbf{\tilde{f}(x+r v)=\hat{f}(x)+r \tilde{f}(v)}$ for all $\mathbf{x}$ of $\mathbf{\hat{H}}$ and $\mathbf{r}$ of $\mathbf R$. 
    It is enough to find $\mathbf w$ of $\mathbf R$ such that $\mathbf{\hat{f}(x)+r w \leqslant g(x+r v)}$ for all $\mathbf x$ of $\mathbf{\hat{H}}$ and $\mathbf{r}$ of $\mathbf R$. In this case, there is $(a_1,a_2,a_3)\in p(1)$ such that $\mathbf{r}|a_1 > \mathbf{0}|a_1$, $\mathbf{r}|a_2<\mathbf{0}|a_2$ and $\mathbf{r}|a_3=\mathbf{0}|a_3$. 
    Thus
    \begin{align}
    \label{glei1} \mathbf{w}|a_1&\leqslant \left(\mathbf{\frac{1}{r}[g(x+rv)-\hat{f}(x)]}\right)|a_1,\\
    \label{glei2}    \mathbf{w}|a_2&\leqslant \left(\mathbf{-\frac{1}{r}[\hat{f}(x)- g(x+rv)}]\right)|a_2,\\
    \label{glei3}    \mathbf{\hat{f}(x)}|a_3&\leqslant \mathbf{g(x)}|a_3,
    \end{align}
for every $\mathbf{x}$ of $\mathbf{\hat{H}}$. 
The relation \eqref{glei3} is fulfilled by the definition of $\mathbf{\hat{f}}$. 
By inspection, the relations \eqref{glei1} and \eqref{glei2} hold for a conditional element $\mathbf{r}$ of $\mathbf{R}$ if, and only if,
\begin{equation}
 \mathbf{\frac{1}{p} [\hat{f}(x)-g(x-p v)]\leqslant w\leqslant \frac{1}{q} [\ g(y+q v)-\hat{f}(y)]}\label{glei4}
\end{equation}
for all conditional elements $\mathbf{x,y}$ of $\mathbf{\hat{H}}$ and $\mathbf{p,q}$ of $\mathbf{R}_{++}$. 
Thus \eqref{glei4} is fulfilled for some conditional element $\mathbf{r}$ of $\mathbf{R}$ which is the desired contradiction. 
\end{proof}
\begin{definition}
    A conditional vector space $\mathbf{X}$ endowed with a conditional topology $\mathcal{T}$  is  a 
    \emph{conditional topological vector space} if the conditional functions  $+:\mathbf{X}\times \mathbf{X}\to \mathbf{X}$ and $\cdot:\mathbf{R}\times \mathbf{X}\to \mathbf{X}$ are conditionally continuous. 
    We call a conditional subset $\mathbf{Y}$ of $\mathbf{X}$ on $1$ \emph{conditionally $\mathcal{T}$-bounded} if for every conditional neighborhood $\mathbf{U}$ of $\mathbf{0}$ there exists $\mathbf{r>0}$ such that $\mathbf{Y}\sqsubseteq \mathbf{r} \mathbf{U}$. 
    For a conditional topological vector space $\mathbf{X}$, denote by $\mathbf{X^\prime}$ the conditional vector space of all conditionally continuous and linear functions $\mathbf{f:\mathbf{X}\to \mathbf{R}}$.

    A conditional topological vector space $\mathbf{X}$ is \emph{conditionally locally convex} if $\mathbf{X}$ has a conditional neighborhood base of $\mathbf{0}$ consisting of conditionally convex sets. 
\end{definition}
Inspection shows that a conditional topological vector space has a conditional neighborhood base of $\mathbf{0}$ consisting of conditionally closed, absorbing and circled conditional sets. 

Separation theorems for locally convex topological $L^0$-modules are proved in \citep[Theorem 2.6 and 2.8]{kupper03}. 
The following theorem is the respective analogue for conditional locally convex topological vector spaces. Its proof technique is similar to \cite{kupper03}. 
\begin{theorem}\label{thm:condsep}
    Let $\mathbf{X}$ be a conditional locally convex topological vector space and $\mathbf{C}_1,\mathbf{C}_2$  two conditionally convex subsets of $\mathbf{X}$ on $1$ such that $\mathbf{C}_1 \sqcap \mathbf{C}_2=\mathbf{X}|0$.
    \begin{enumerate}[label=(\roman*)]
        \item If $\mathbf{C}_1$ is conditionally open, then there exists a conditionally continuous linear function $\mathbf{f}:\mathbf{X} \to \mathbf{R}$  such that
            \begin{equation*}
                \mathbf{f(x)}< \mathbf{f(y)}\quad \text{for every conditional element }\mathbf{x} \text{ of }\mathbf{C}_1\text{ and }\mathbf{y} \text{ of }\mathbf{C}_2.
                \label{eq:condsep01}
            \end{equation*}
        \item If $\mathbf{C}_1$ is conditionally compact and $\mathbf{C}_2$ conditionally closed, then there exists a conditionally continuous linear function $\mathbf{f}:\mathbf{X} \to \mathbf{R}$ and a conditional element $\mathbf{r}$ of $\mathbf{R}_{++}$ such that 
            \begin{equation*}
                \mathbf{f(x)}+\mathbf{r} <\mathbf{f(y)}\quad \text{for every conditional element }\mathbf{x} \text{ of }\mathbf{C}_1\text{ and }\mathbf{y} \text{ of }\mathbf{C}_2.
                \label{eq:condsep02}
            \end{equation*}
    \end{enumerate}
\end{theorem}
\begin{definition}
    Two conditional vector spaces $\mathbf{X}$ and $\mathbf{Y}$ form a \emph{conditional dual pair}, denoted by  
    $\langle \mathbf{X},\mathbf{Y}\rangle$, if there exists a conditional  function  
    $\langle \cdot,\cdot\rangle :\mathbf{X}\times \mathbf{Y} \to \mathbf{R}$ such that 
    \begin{enumerate}[label=\textit{(\roman*)}]
        \item if $\mathbf{x} \mapsto \langle \mathbf{x},\mathbf{y}\rangle $ for every fixed conditional element $\mathbf{y}$ of $\mathbf{Y}$ and $\mathbf{y} \mapsto \langle \mathbf{x},\mathbf{y}\rangle$ for every fixed conditional element $\mathbf{x}$ of $\mathbf{X}$ are  conditionally linear;
        \item if $\langle \cdot, \mathbf{y}\rangle= \mathbf{0}$ and $\langle \mathbf{x},\cdot\rangle = \mathbf{0}$ imply $\mathbf{y}=\mathbf{0}$ and $\mathbf{x}=\mathbf{0}$, respectively. 
    \end{enumerate}
\end{definition}
For a conditional dual pair $\langle \mathbf{X},\mathbf{Y}\rangle$, we denote by $\sigma(\mathbf{X},\mathbf{Y})$ and $\sigma(\mathbf{Y},\mathbf{X})$ the conditional initial topologies for $(\langle \cdot ,\mathbf{y}\rangle)_{\mathbf{y} \text{ of }\mathbf{Y}}$ and $(\langle \mathbf{x}, \cdot \rangle)_{\mathbf{x} \text{ of }\mathbf{X}}$, respectively.  
It can be checked that $(\mathbf{X},\sigma(\mathbf{X},\mathbf{Y}))^\prime=\mathbf{X}^\prime$. 
For a conditional locally convex topological vector space $\mathbf{X}$, it follows from Theorem \ref{thm:hahnbanach} that $\langle \mathbf{X},\mathbf{X^\prime}\rangle$ is a conditional dual pair with respect to the conditional function $(\mathbf{x,x^\prime})\mapsto \mathbf{x^\prime(x)}$. 
\begin{definition}
Let $\langle \mathbf{X},\mathbf{Y} \rangle $ be a conditional dual pair and $\mathbf{Z}$ a conditional subset of $\mathbf{X}$ on $1$. 
The conditional subset $\mathbf{Z^\circ}:=\Set{\mathbf{y}: \mathbf{\langle x,y \rangle\leqslant 1} \text{ for all } \mathbf{x} \text{ of } \mathbf{Z}}$\footnote{Note that $\mathbf{Z}^\circ$ is on $1$ since $\mathbf{0}$ is a conditional element of it.} of $\mathbf{Y}$ is called the \emph{conditional polar} of $\mathbf{Z}$. 
\end{definition}
\begin{proposition}\label{p:polar}
Let $\langle \mathbf{X},\mathbf{Y}\rangle$ be a conditional dual pair, $\mathbf{Z}$ and $\mathbf{W}$  two conditional subsets of $\mathbf{X}$ on $1$, and $(\mathbf{Z}_i)$  a non-empty family of conditional subsets of $\mathbf{X}$ each of which is on $1$. Then it holds:
    \begin{enumerate}[label=\textit{(\roman*)}]
        \item If $\mathbf{Z}\sqsubseteq \mathbf{W}$, then $\mathbf{W}^\circ \sqsubseteq \mathbf{Z}^\circ$.\label{pol:01}
        \item If $\mathbf{r}$ is a conditional element of $\mathbf{R}^{\ast}=\{\mathbf{0}\}^\sqsubset$, then $(\mathbf{rZ})^\circ=\mathbf{(1/r)} \mathbf{Z}^\circ$.\label{pol:02} 
        \item $(\sqcup \mathbf{Z}_i)^\circ=\sqcap \mathbf{Z}_i^\circ$. \label{verschnittpolar}
        \item $\mathbf{Z}^\circ$ is conditionally convex, $\sigma(\mathbf{Y},\mathbf{X})$-closed and $\mathbf{0}$ is a conditional element of it.\label{polarabg} 
    \end{enumerate}
    \end{proposition}
\begin{proof}
The assertions \ref{pol:01}, \ref{pol:02} and \ref{verschnittpolar} are immediate from the definitions. 
Proposition \ref{stetigmitnetz} implies assertion \ref{polarabg}. 
\end{proof}
The previous proposition together with Theorem \ref{thm:condsep} yield a conditional Bipolar theorem.  
\begin{theorem}\label{thm:bipl}
    Let $\langle \mathbf{X},\mathbf{Y} \rangle$ be a conditional dual pair and $\mathbf{Z}$  a conditional subset of $\mathbf{X}$ on $1$. 
    Then $\mathbf{Z}^{\circ \circ}=\Cl(\co(\mathbf{Z}\sqcup \mathbf{0}))$. 
            Thus if $\mathbf{Z}$ is conditionally convex, $\sigma(\mathbf{X},\mathbf{Y})$-closed and $\mathbf{0}$ is a conditional element of it, then $\mathbf{Z}^{\circ \circ}=\mathbf{Z}$. 
\end{theorem}
Next we prove a conditional Banach-Alaoglu theorem. 
\begin{theorem}
Let $\mathbf{X}$ be a conditional locally convex topological vector space and $\mathbf{U}$  a conditional neighborhood of $\mathbf{0}$. 
Then $\mathbf{U}^\circ$ is conditionally $\sigma(\mathbf{X}^\prime,\mathbf{X})$-compact. 
    \label{thm:banach-alaoglu}
\end{theorem}
\begin{proof}
The conditional $\sigma(\mathbf{X^\prime,X})$-topology is the conditional topology of conditional pointwise convergence in $\mathbf{X}$ and a conditional relative topology of the conditional product topology $\mathcal{T}$ on $\mathbf{R}^\mathbf{X}$.\footnote{As in the classical case, one identifies $\mathbf{X}^\prime$ with a conditional subspace of $\mathbf{X}^\mathbf{R}$.} 
Without loss of generality, we may assume that $\mathbf{U}$ is conditionally convex and circled due to Proposition \ref{p:polar}. 
The conditional function 
\begin{equation*}
\mathbf{\Phi}:(\mathbf{X}^\prime,\sigma(\mathbf{X^\prime,X}))\to (\mathbf{R}^\mathbf{X},\mathcal{T}), \quad \mathbf{\Phi}(\mathbf{x}^\prime)(\mathbf{x})=\langle \mathbf{x}^\prime,\mathbf{x}\rangle
\end{equation*}
is conditionally injective and its conditional inverse function 
\begin{equation*}
\mathbf{\Phi}^{-1}:(\mathbf{R}^\mathbf{X},\mathcal{T})\to (\mathbf{X}^\prime,\sigma(\mathbf{X^\prime,X}))
\end{equation*}
conditionally continuous. 
By Proposition \ref{p:compactsets}, it is enough to show that $\mathbf{\Phi(U^\circ)}$ is conditionally compact. 
To this end, let $\mathbf{x}$ be a conditional element of $\mathbf{X}$. 
Since $\mathbf{U}$ is conditionally absorbing, there exists $\mathbf{r_x>0}$ such that $\mathbf{x}$ is a conditional element of $\mathbf{r_x U}$. 
For $\mathbf{x}$ being a conditional element of $\mathbf{U}$, choose $\mathbf{r_x \leqslant 1}$. 
Since $\mathbf{U}$ is conditionally circled, it holds $\mathbf{\abs{\langle x^\prime, x\rangle}\leqslant r_x}$ for all conditional elements $\mathbf{x^\prime}$ of $\mathbf{U^\circ}$. 
Put $\mathbf{K_x}=\{\mathbf{r}: \mathbf{\abs{r}\leqslant r_x}\}$. 
It follows from Theorem \ref{thm:metriccompat} that $\mathbf{K_x}$ is conditionally compact. 
By Theorem \ref{thm:tychonoff}, 
\begin{equation*}
\mathbf{K}:=\{\mathbf{f} \text{ of } \mathbf{R^X}: \mathbf{f(x)} \text{ of } \mathbf{K_x} \text{ for all } \mathbf{x} \text{ of } \mathbf{X}\}
\end{equation*}
is conditionally compact. 
By Proposition \ref{p:closedcompact}, it remains to show that $\mathbf{\Phi(\mathbf{U}^\circ)}$ is conditionally closed. 
Let $\mathbf{f}$ be a conditional element of $\Cl(\mathbf{\Phi(\mathbf{U}^\circ)})$. 
Since the conditional pointwise limit of a conditional net of conditional linear functions is conditionally linear (due to the conditional continuity of conditional addition and scalar multiplication), $\mathbf{f}$ is conditionally linear. Since $\mathbf{r_x\leqslant 1}$ for $\mathbf{x}$ of $\mathbf{U}$, it follows that $\mathbf{f(U)}\sqsubseteq \{\mathbf{r}: \abs{\mathbf{r}}\leqslant \mathbf{1}\}$. 
Therefore, $\mathbf{f}$ is a conditional element of $\mathbf{\Phi(\mathbf{X}^\prime)}$. 
The remaining assertion follows from the fact that $\mathbf{U^\circ}$ is conditionally $\sigma(\mathbf{X^\prime,X})$-closed. 
    \end{proof}
\begin{definition}
    Let $\mathbf{X}$ be a conditional vector space.
    A \emph{conditional norm} is a conditional function $\norm{\cdot}:\mathbf{X}\to \mathbf{R}_+$ such that
    \begin{enumerate}[label=\textit{(\roman*)}]
        \item\label{cond:norm01} $\norm{\mathbf{x}}=\mathbf{0}$ if, and only if, $\mathbf{x}=\mathbf{0}$,
        \item\label{cond:norm02} $\norm{\mathbf{r} \mathbf{x}}=\abs{\mathbf{r}}\norm{\mathbf{x}}$ for all conditional elements $\mathbf{x}$ of $\mathbf{X}$ and $\mathbf{r}$ of $\mathbf{R}$,
        \item\label{cond:norm03} $\norm{\mathbf{x}+\mathbf{y}}\leqslant \norm{\mathbf{x}}+\norm{\mathbf{y}}$ for all conditional elements $\mathbf{x},\mathbf{y}$ of $\mathbf{X}$.
    \end{enumerate}
    A conditional vector space together with a conditional norm is called a \emph{conditional normed vector space}. 
\end{definition}
For a conditional norm $\norm{\cdot}:\mathbf{X}\to \mathbf{R}$, the conditional function $\mathbf{d}:\mathbf{X}\times \mathbf{X}\to \mathbf{R}$ 
defined by $\mathbf{d}(\mathbf{x},\mathbf{y})=\norm{\mathbf{x}-\mathbf{y}}$ is a conditional metric. 
A conditional normed vector space $(\mathbf{X},\norm{.})$ is called a \emph{conditional Banach space} if $(\mathbf{X},\mathbf{d})$ is conditionally complete.  
For a conditionally linear operator $\mathbf{T}:\mathbf{X}\to \mathbf{R}$, the conditional operator norm $\norm{\mathbf{T}}^\prime$ is defined by $\norm{\mathbf{T}}^\prime:=\sup\Set{\abs{\mathbf{T(x)}}:\mathbf{x} \text{ of } \mathbf{X}, \norm{\mathbf{x}}=\mathbf{1}}$. 
By inspection, if $\mathbf{X}$ is conditionally Banach, then $(\mathbf{X^\prime},\norm{\cdot}^\prime)$ is so.  

We close this paper by a conditional Krein-\v{S}mulian theorem. 
A conditional version of this theorem for modules over $L^\infty$ is proved in Eisele and Taieb \cite[Theorem 11.1]{eisele13}. 
For $\mathbf{r>0}$, let  
\begin{align*}
\mathbf{C^\prime_r}&=\{\mathbf{x^\prime}: \norm{\mathbf{x^\prime}}^\prime\leqslant \mathbf{r}\},\quad
\mathbf{C^\prime_r}(\mathbf{y^\prime})=\{\mathbf{x^\prime}: \norm{\mathbf{x^\prime-y^\prime}}^\prime\leqslant \mathbf{r}\} \quad \text{ for each $\mathbf{y^\prime}$ of $\mathbf{X^\prime}$}. 
\end{align*}
In particular, we denote by $\mathbf{C^\prime}:=\mathbf{C^\prime_1}$ and $\mathbf{C}:=\{\mathbf{x}: \norm{\mathbf{x}}\leqslant 1\}$. 
\begin{theorem}
    Let $\mathbf{X}$ be a conditional Banach space and $\mathbf{Y}$  a conditionally convex subset of $\mathbf{X}^\prime$ on $1$. Then $\mathbf{Y}$ is conditionally $\sigma(\mathbf{X}^\prime,\mathbf{X})$-closed if, and only if,
    $\mathbf{Y}\sqcap \mathbf{C_n^\prime}$ is conditionally $\sigma(\mathbf{X}^\prime,\mathbf{X})$-closed for all conditional elements $\mathbf{n}$ of $\mathbf{N}$. 
    \label{thm:Banach-Dieudonne}
\end{theorem}
\begin{proof} 
If $\mathbf{Y}$ is conditionally $\sigma(\mathbf{X^\prime,X})$-closed, then its conditional intersection with each $\mathbf{C_{n}^\prime}$ is so, since $\mathbf{C_{n}^\prime}$ 
is conditionally $\sigma(\mathbf{X^\prime,X})$-closed. 

As for the converse, note first that $\mathbf{Y}$ is conditionally norm-closed.
    Indeed, for a conditional sequence $(\mathbf{x^\prime_n})$ of conditional elements of $\mathbf{Y}$ such that $\mathbf{x^\prime_n}\to \mathbf{x^\prime}$, it follows from the conditional triangle inequality that there is some $\mathbf{m}$ of $\mathbf{N}$ such that $(\mathbf{x^\prime_n})$ is a conditional sequence of conditional elements of $\mathbf{Y}\sqcap \mathbf{C_m^\prime}$.
    Since conditional norm-convergence implies conditional $\sigma(\mathbf{X^\prime,X})$-convergence\footnote{This follows from $\abs{\langle \mathbf{x_n}, \mathbf{x^\prime}\rangle -\langle \mathbf{x,x^\prime}\rangle}\leqslant \norm{\mathbf{x^\prime}}^\prime
    \norm{\mathbf{x_n-x}}$.}  
    and since $\mathbf{x_n^\prime}$ is a conditional element of $\mathbf{Y}\sqcap \mathbf{C_m^\prime}$ which is conditionally $\sigma(\mathbf{X^\prime,X})$-closed by assumption, $\mathbf{x^\prime}$ is a conditional element of $\mathbf{Y}$. 
    Further $\mathbf{Y}\sqcap \mathbf{C^\prime_r(y^\prime)}$ is conditionally $\sigma(\mathbf{X}^\prime,\mathbf{X})$-closed for all conditional elements $\mathbf{r}$ of $\mathbf{R_{++}}$ and $\mathbf{y^\prime}$ of $\mathbf{X^\prime}$. 
Indeed, there exists $\mathbf{n}$ of $\mathbf{N}$ such that $\mathbf{C^\prime_r(y^\prime)}\sqsubseteq \mathbf{n}\mathbf{C^\prime}$. 
By assumption, $\mathbf{Y}\sqcap \mathbf{C^\prime_r(y^\prime)}=(\mathbf{Y}\sqcap \mathbf{n}\mathbf{C^\prime})\sqcap \mathbf{C^\prime_r(y^\prime)}$ is conditionally $\sigma(\mathbf{X^\prime,X})$-closed. 

    Up to conditional addition of $\mathbf{Y}$ by a conditional element, which is a conditionally continuous operation, we may assume that $\mathbf{0}$ is a conditional element of $\mathbf{Y}$. 
For each conditional element $\mathbf{n}$ of $\mathbf{N}$, put $\mathbf{Y_n}=\mathbf{Y\sqcap 2^n C^\prime}$. 
Then $\mathbf{Y_n}$ is conditionally convex, $\sigma(\mathbf{X^\prime,X})$-closed and $\mathbf{0}$ is a conditional element of $\mathbf{Y_n}$. 
It follows from Theorem \ref{thm:bipl} that $\mathbf{Y_n}=\mathbf{Y_n^{\circ\circ}}$. 
Put $\mathbf{Z}=\sqcap \mathbf{Y_n^\circ}$. Note that $\mathbf{Z}$ is on $1$ since $\mathbf{0}$ is a conditional element of each $\mathbf{Y_n^\circ}$. 
We will show that $\mathbf{Y}=\mathbf{Z^\circ}$ which will imply that $\mathbf{Y}$ is conditionally $\sigma(\mathbf{X^\prime,X})$-closed by Proposition \ref{p:polar}. 
Since $\mathbf{Z}\sqsubseteq \mathbf{Y_n^\circ}$, it follows that $\mathbf{Y_n}\sqsubseteq \mathbf{Z^\circ}$ from Proposition \ref{p:polar} which in turn implies $\mathbf{Y}=\sqcup \mathbf{Y_n}\sqsubseteq \mathbf{Z^\circ}$. 
The following steps will establish $\mathbf{Z^\circ}\sqsubseteq \mathbf{Y}$. 
    \begin{enumerate}[label=\textit{Step \arabic* :},fullwidth,ref=\textit{Step \arabic*}]
        \item\label{s1} We show that $\mathbf{Y_n^\circ}\sqsubseteq \mathbf{Y_{n+1}^\circ}+ \mathbf{2^{-n}}\mathbf{C}$ for every $\mathbf{n}$ of $\mathbf{N}$. To this end, let $\mathbf{x}$ be a conditional element of $(\mathbf{Y_{n+1}^\circ + 2^{-n} C})^\sqsubset$. By Theorem \ref{thm:condsep}, it holds $\mathbf{x^\prime}(\mathbf{x})\geqslant \mathbf{1}\geqslant \sup\{\mathbf{x^\prime(y)}: \mathbf{y} \text{ of } \mathbf{Y^\circ_{n+1}} + \mathbf{2^{-n} C}\}$ for some $\mathbf{x^\prime}$ of $\mathbf{X^\prime}$. 
        Since $\mathbf{2^{-(n+1)} C}\sqsubseteq \mathbf{Y^\circ_{n+1}}$ and $\mathbf{x^\prime(3\cdot 2^{-(n+1)}C)}=\mathbf{x^\prime(2^{-(n+1)}C + 2^{-n}C)}\sqsubseteq \mathbf{x^\prime(\mathbf{Y}^{\circ}_{n+1}+2^{-n}C)}$ by \eqref{f:021}, it holds $\norm{\mathbf{x^\prime}}^\prime\leqslant \mathbf{\frac{2}{3}2^n}$. 
        Since $$\sup\{\mathbf{y}: \mathbf{y}\text{ of } \mathbf{Y_{n+1}^\circ + 2^{-n} C}\}=\sup\{\mathbf{y}: \mathbf{y}\text{ of } \mathbf{Y_{n+1}^\circ}\}+\sup\{\mathbf{y}: \mathbf{y}\text{ of } \mathbf{2^{-n} C}\},$$ one has $\sup\{\mathbf{y}: \mathbf{y}\text{ of } \mathbf{Y_{n+1}^\circ}\}\leqslant \mathbf{1 - 2^{-n}}\norm{\mathbf{x^\prime}}^\prime$. 
        Choose $\mathbf{r}$ to be the conditional minimum of $\mathbf{1/3}$ and $\mathbf{2^{-n}}\norm{\mathbf{x^\prime}}^\prime$.
        Put $\mathbf{y^\prime}:=\mathbf{\frac{1}{1-r}}\mathbf{x^\prime}$. From $\mathbf{r}<\mathbf{2^{-n}}\norm{\mathbf{x^\prime}}^\prime$ and $\mathbf{\sup}\{\mathbf{y}: \mathbf{y}\text{ of } \mathbf{Y_{n+1}^\circ}\}\leqslant \mathbf{1 - 2^{-n}}\norm{\mathbf{x^\prime}}^\prime$, it follows that $\sup\{\mathbf{y}: \mathbf{y}\text{ of } \mathbf{Y_{n+1}^\circ}\}\leqslant \mathbf{1}$, that is, $\mathbf{y^\prime}$ is a conditional element of $\mathbf{Y_{n+1}^{\circ\circ}}=\mathbf{Y_{n+1}}$. 
        By the choice of $\mathbf{r}$, we get $\norm{\mathbf{y^\prime}}\leqslant \mathbf{2^n}$, and therefore $\mathbf{y^\prime}$ is a conditional element of $\mathbf{Y_n}$. Since $\mathbf{y^\prime}(\mathbf{x})>\mathbf{1}$, it follows that $\mathbf{x}$ is a conditional element of $(\mathbf{Y_n}^\circ)^\sqsubset$. 
       \item\label{s2} It is shown that $\mathbf{Y_n^\circ} \sqsubseteq \mathbf{Z}+\mathbf{2^{-(n-1)}C}$ for every conditional element $\mathbf{n}$ of $\mathbf{N}$. 
            For each $\mathbf{n}$ of $\mathbf{N}$, we pick inductively, using the result of \ref{s1}, $\mathbf{x_{n+1}},\mathbf{x_{n+2}},\ldots$ such that $\mathbf{x_m}$ is a conditional element of $\mathbf{Y_{m}^\circ}$ and $\norm{\mathbf{x_m-x_{m+1}}}\leqslant \mathbf{2^{-m}}$ for $\mathbf{m=n,n+1,\ldots}$. 
            For a conditional element $\mathbf{m}$ of $\mathbf{N}$ with $\mathbf{m=\sum m_i}|a_i$ where $(m_i)$ is a family in $\{n,n+1,\ldots\}$ and $(a_i)\in p(1)$, put $\mathbf{x_m}:=\sum \mathbf{x_{m_i}}|a_i$.  
            Since $(\mathbf{Y_n^\circ})$ is a conditional family, each $\mathbf{x_m}$ is a conditional element of $\mathbf{Y_m^\circ}$.
            By construction, $(\mathbf{x_m})$ is a conditional Cauchy sequence. 
            By the conditional completeness of $\mathbf{X}$, there exists a conditional element $\mathbf{x}$ of $\mathbf{X}$ such that $\mathbf{x_m}\to \mathbf{x}$. 
            Each $\mathbf{Y_n^\circ}$ is conditionally norm-closed. 
            Thus $\mathbf{x}$ is a conditional element of $\mathbf{Z}$ since $(\mathbf{x_m})$ is a conditional element of each $\mathbf{Y_m^\circ}$. 
            Moreover, it holds $\norm{\mathbf{x-x_n}}\leqslant \sum_{j=n}^\infty \mathbf{\frac{1}{2^j}}\leqslant \mathbf{\frac{1}{2^{n-1}}}$.  
        \item\label{s3} We show that $\mathbf{Y}=\sqcap_{\mathbf{r>0}}(\mathbf{1+r})\mathbf{Y}$. 
            Since $\mathbf{Y}$ is conditionally convex and $\mathbf{0}$ is a conditional element of $\mathbf{Y}$, the conditional inclusion $\sqsubseteq$ is immediate. 
            Conversely, for $\mathbf{x^\prime}$ of $\sqcap_{\mathbf{r>0}}(\mathbf{1+r})\mathbf{Y}$, setting $\mathbf{x_n^{\prime}=n/(1+n)x^\prime}$ for $\mathbf{n}$ of $\mathbf{N}$, defines a conditional sequence of conditional elements of $\mathbf{Y}$ such that $\norm{\mathbf{x^\prime-x_n^\prime}}^\prime\leqslant \norm{\mathbf{x^\prime}}^\prime/(\mathbf{n+1})\rightarrow \mathbf{0}$. 
            Hence $\mathbf{x_n\rightarrow x}$, and thus $\mathbf{x}$ is a conditional element of $\mathbf{Y}$ since $\mathbf{Y}$ is conditionally norm-closed. 
        \item We show that $\mathbf{Z}^\circ \sqsubseteq \sqcap_{\mathbf{r>0}}(\mathbf{1+r})\mathbf{Y}$. 
            It holds $\mathbf{x+y=(1+r)(\frac{x}{1+r}+(1-\frac{1}{1+r})\frac{y}{r})}$ for all $\mathbf{x},\mathbf{y}$ of $\mathbf{X}$ and some $\mathbf{r>0}$.  
            From \ref{s2} it follows that $\mathbf{Y_n^\circ}\sqsubseteq (\mathbf{1+r})\co(\mathbf{Z}\sqcup \mathbf{2^{n-1}/r C})$. 
            It follows from Proposition \ref{p:polar} that $\mathbf{1/(1+r) (Z^\circ \sqcap 2^{n-1}r C^\prime)}\sqsubseteq \mathbf{Y_n}\sqsubseteq \mathbf{Y}$ for every $\mathbf{n}$ of $\mathbf{N}$. 
            By taking the conditional union over all $\mathbf{n}$ of $\mathbf{N}$, we obtain $\mathbf{1/(1+r) Z^\circ} \sqsubseteq \mathbf{Y}$ for every $\mathbf{r>0}$, and thus $\mathbf{Z^\circ} \sqsubseteq \sqcap_{\mathbf{r>0}} (\mathbf{1+r})\mathbf{Y}=\mathbf{Y}$ by means of \ref{s3}. 
    \end{enumerate}
\end{proof}

\end{document}